\date{}
\title{Dynamical sensitivity of the infinite cluster
in critical percolation}
\author{Yuval Peres, Oded Schramm, Jeffrey E. Steif
}

\documentclass[12pt,hdvips,naturalnames]{article}
\usepackage{amsmath}
\usepackage{amssymb}
\usepackage{amsthm}
\usepackage{amsfonts}
\usepackage{graphicx}
\usepackage{pstricks}
\input{epsf.sty}

\newif\ifhyper\IfFileExists{hyperref.sty}{\hypertrue}{\hyperfalse}

\ifhyper\usepackage[hypertex]{hyperref}\fi

\newif\ifdraft
\drafttrue

\long\def\comment#1{}
\long\def\oldVersion#1{}

\numberwithin{equation}{section}
\numberwithin{figure}{section}

\newtheorem{theorem}{Theorem}
\numberwithin{theorem}{section}

\newtheorem{lemma}[theorem]{Lemma}
\newtheorem{proposition}[theorem]{Proposition}

\theoremstyle{remark}
\theoremstyle{remark}

\def\QED{\qed\medskip}

\newcommand{\R}{\mathbb{R}}

\newcommand{\N}{\mathbb{N}}

\def\calC{{\cal C}}

\def\bPsi{\mbox{\boldmath $\Psi$}}

\def\FCbetter(#1,#2){{ #1 \leftrightarrow #2 }}
\def\percrt{\FCbetter(\rho,\infty)}
\def\FCnot(#1,#2){{ 0 \not \stackrel{#2}{\leftrightarrow} #1 }}
\def\FCC(#1,#2,#3){{ #1 \stackrel{#2}{\leftrightarrow} #3 }}
\def\FCCnot(#1,#2,#3){{ #1 \stackrel{#2}{\not\leftrightarrow} #3 }}
\def\FCbetterTree(#1,#2){{ #1 \mapsto #2 }}
\def\FCnotTree(#1,#2){{ 0 \not \stackrel{#2}{\mapsto} #1 }}
\def\FCCTree(#1,#2,#3){{ #1 \stackrel{#2}{\mapsto} #3 }}
\def\FCCnotTree(#1,#2,#3){{ #1 \stackrel{#2}{\not\mapsto} #3 }}
\def\tmG{\tilde\mG}
\def\mG{{\mathcal G}}

\def\WW{Q}
\def\Ito/{It\^o}
\def\tree{T}
\def\perctime{\mathcal Q}
\def \eps {\epsilon}
\def \P {{\bf P}}
\def\md{\mid}
\def\Bb#1#2{{\def\md{\bigm| }#1\bigl[#2\bigr]}}
\def\BB#1#2{{\def\md{\Bigm| }#1\Bigl[#2\Bigr]}}
\def\Bs#1#2{{\def\md{\mid}#1[#2]}}
\def\Pb{\Bb\P}
\def\Eb{\Bb\E}
\def\PB{\BB\P}
\def\EB{\BB\E}
\def\ew#1{w_{#1}}

\def\Es{\Bs\E}

\def \E {{\bf E}}

\def\closure{\overline}
\def\ev#1{{\mathcal{#1}}}
\def \proof {{ \medbreak \noindent {\bf Proof.} }}
\def\proofof#1{{ \medbreak \noindent {\bf Proof of #1.} }}

\def\bl{\bigl}\def\br{\bigr}\def\Bl{\Bigl}\def\Br{\Bigr}

\def\One{{\mathbf 1}}

\def\noopsort#1{}

\begin{document}
\maketitle

 \begin{abstract}
In dynamical percolation, the status of every bond is refreshed according to an independent Poisson clock. 
For graphs which do not percolate at criticality, the dynamical sensitivity of this property was analyzed 
extensively in the last decade. Here we focus on graphs which percolate at  criticality,
and investigate the dynamical sensitivity of the infinite cluster.
We first give two examples of bounded
degree graphs, one which percolates
for all times at criticality and one which has exceptional times of
nonpercolation.
We then make a nearly complete analysis of this question for spherically symmetric trees
with spherically symmetric edge probabilities
bounded away from $0$ and $1$. 
One interesting regime occurs when the expected number of vertices at the 
$n$th level that connect to the root at a fixed time is
of  order $n(\log n)^\alpha$.
R.\ Lyons (1990) showed  that  at a fixed time, there is an infinite cluster a.s.\ if and only if $\alpha >1$.
We prove that  the probability that there is an infinite cluster at all  times is 1
if   $\alpha > 2$, while this probability is 0 if $1<\alpha \le 2$. 
Within the regime where a.s.\ there is an infinite cluster at all times,
there is yet another type of ``phase transition'' in the
behavior of the process: if the expected number of
vertices at the $n$th level
connecting to the root at a fixed time
is of order $n^\theta$ with $\theta > 2$,
then the number of connected components of the set of times
in $[0,1]$ at which the root does not percolate
is finite a.s., while if $1<\theta < 2$,  then the number of such
components is infinite with positive probability.
\end{abstract}



\section {Introduction}

Consider bond percolation on an infinite connected locally finite graph
$G$, where for some $p\in[0,1]$, each edge (bond) of $G$ is, independently of
all others, open with probability $p$ and closed with probability $1-p$.
Write $\pi_p$ for this product measure.
Some of the main questions in percolation theory (see \cite{Grimmett})
deal with the possible existence of infinite connected components
(clusters) in the random subgraph of $G$
consisting of all sites and all open edges.
Write $\calC$ for the event
that there exists such an infinite cluster. By Kolmogorov's 0-1 law,
the probability of $\calC$ is, for fixed $G$ and $p$, either 0 or 1.
Since $\pi_p(\calC)$ is nondecreasing in $p$, there
exists a critical probability $p_c=p_c(G)\in[0,1]$ such that
\[
\pi_p(\calC)=\left\{
\begin{array}{ll}
0 & \mbox{for } p<p_c \\
1 & \mbox{for } p>p_c.
\end{array} \right.
\]
At $p=p_c$, we can have either $\pi_p(\calC)=0$ or $\pi_p(\calC)=1$, 
depending on $G$.

H\"{a}ggstr\"{o}m, Peres and Steif~\cite{HPS} initiated the study of
dynamical percolation.
In this model, with $p$ fixed, the edges of $G$ switch back
and forth according to independent 2 state
continuous time Markov chains where closed switches to open at rate $p$
and open switches to closed at rate $1-p$. Clearly, $\pi_p$ is a stationary
distribution for this Markov process. The general question studied in
\cite{HPS} was whether, when we start with distribution $\pi_p$,
there could exist atypical times at which the percolation
structure looks markedly different than that at a fixed time. As the results in
\cite{HPS} suggest, it is most interesting to consider things at criticality;
that is, when $p=p_c$.

Write $\bPsi_p$ for the underlying probability measure
of this Markov process, and write $\calC_t$ for
the event that there is an infinite cluster of open edges (somewhere in
the graph) at time $t$.

There have been a number of papers on dynamical percolation after
\cite{HPS}, namely  \cite{PS}, \cite{K} and \cite{SS},
but all of the results (except one, see the comment after
Theorem~\ref{th:newexamples}) in these papers have been concerned with the case
where the graph does not
percolate at criticality (and for which there may or may not exist exceptional
times). The present paper deals with the case where the graph percolates at criticality
at a fixed time.

\medskip
Our first theorem gives examples where exceptional times exist,
and other examples where they do not exist.

\begin{theorem} \label{th:newexamples}

\smallskip\noindent
(i). There is a bounded degree graph which, at criticality, percolates at all times; i.e.,
\begin{equation}
\bPsi_{p_c}(\, \calC_t \, \mbox{ occurs for all  } \, t \, )=1.
\end{equation}

\smallskip\noindent
(ii). There is a bounded degree graph which percolates at criticality but has exceptional times, i.e.,
\begin{equation}
\bPsi_{p_c}(\neg\, \calC_t \, \mbox{ occurs for some  } \, t \, )=1.
\end{equation}
\nonumber
\end{theorem}

\medskip\noindent
{\bf Remarks:}
An example of an unbounded degree graph which percolates
at criticality but for which there are exceptional times of nonpercolation
can be found in \cite{HPS}.

Although Theorem~\ref{th:newexamples} follows from our Theorem~\ref{th:sstrees1}
below, we find it instructive to treat it separately, since the proof
is easier and self-contained.

\medbreak

We now discuss spherically symmetric trees with spherically symmetric edge
probabilities. These are trees in which every vertex on a given level has
the same number of offsprings and the edge probabilities may vary but are
constant on a given level.

Denote the root of the tree by $\rho$,
the edge probability for edges going from level $n-1$ to level $n$ by $p_n$,
the set of vertices at level $n$ by $T_n$ and the subtree of $T$ rooted at
some vertex $x$ by $T^x$.

\medskip\noindent
{\bf  Standing assumption:}
We assume throughout the paper that $0< \inf_n p_n \le \sup_n p_n < 1$.

\medskip
By a result of R.\ Lyons (\cite{Ly}),
percolation occurs (at a fixed time) if and only if
$$
\sum_n{\frac{(\prod_{i=1}^n p_i)^{-1}}{|T_n|}} <\infty.
$$
If we let $W_n:=|\{x\in T_n: \FCbetter(\rho,x)\}|$
and $\ew n:=\E[W_n]$, this is equivalent to
\begin{equation} \label{e:perc}
\sum_n{\frac{1}{\ew n}} <\infty.
\end{equation}
In fact, it follows from \cite{Ly} that
\begin{equation} \label{e:lyonsstrong}
P(\FCbetter(\rho,T_{n})) \asymp
\left(\sum_{k=1}^n\frac{1}{\ew k}\right)^{-1}.
\end{equation}
(The relation $\asymp$ means that the ratio between the two sides
is bounded between two positive constants which
may depend on $\inf_n p_n$ and $\sup_n p_n$.)

Dynamical percolation for a graph with edge dependent probabilities is defined
in the obvious way.
To be able to see the crossover between having exceptional times of
nonpercolation and not having such times, we need to look at things
at the right scale. It turns out that the proper parameterization
is to assume that $\ew n\asymp n(\log n)^\alpha$ for some $\alpha>0$.
Lyons' criterion~\eqref{e:perc} easily yields that percolation occurs (at a 
fixed time) if and only if $\alpha > 1$.

\begin{theorem} \label{th:sstrees1}
Consider a spherically symmetric tree with spherically symmetric edge probabilities.

\smallskip\noindent
(i). If
$$
\lim_n \frac{\ew n}{n(\log n)^{\alpha}}=\infty
$$
for some $\alpha > 2$, then there are no exceptional times of  nonpercolation.

\smallskip\noindent
(ii). If
$$
{\ew n}\asymp {n(\log n)^{\alpha}}
$$
for some $1<\alpha \le 2$, then there are exceptional times of nonpercolation.
\end{theorem}

\medskip\noindent
{\bf Remarks:} \\
(1). To see a concrete example, if
we have a tree with $|T_n| \asymp 2^nn(\log n)^\alpha$ and
$p=1/2$ for all edges, then if $\alpha>2$, we are in case (i)
while if $\alpha\le 2$, we are in case (ii). (Note Lyons'
theorem tells us that $p_c=1/2$ in these cases.) \\
(2). The theorem implies that if $\ew n\asymp n^\alpha$ with $\alpha > 1$, then
there are no exceptional times of nonpercolation, while if
$\ew n \asymp n$, then~(\ref{e:perc}) implies that there is no percolation
at a fixed time.
Hence, if we only look at the case where $w_n\asymp  n^\alpha$
for some $\alpha\ge 1$, we do not see the dichotomy we are after.
Rather, Theorem \ref{th:sstrees1} tells us that
one needs to look at a ``finer logarithmic scale'' to see this
``phase transition''.

\medskip
Interestingly, it turns out that even within the regime where
there are no exceptional times of nonpercolation,
there are still two very distinct dynamical behaviors 
of the process.

\begin{theorem}\label{th:dyntrans}
Consider a spherically symmetric tree $T$, with spherically symmetric
edge probabilities. Let $d_j$ denote the number of children that a
vertex in $T_j$ has.

\smallskip\noindent
(i). When $\sum_{k=1}^\infty k\,\ew k^{-1}<\infty$, a.s.\ the set of
times $t\in[0,1]$ at which the root percolates
has finitely many connected components. (This holds for example
if $\ew k\asymp k^\theta$ with $\theta> 2$ as well as for supercritical
percolation on a homogeneous tree.)

\smallskip\noindent
(ii). If $\sup_j d_j<\infty$ and $\ew k\asymp k^\theta$, where
$1<\theta<2$, then with positive probability the set of times $t\in[0,1]$
at which the root percolates has infinitely many connected components.
The same occurs if $\ew k\asymp k(\log k)^\theta$ for any $\theta > 1$.
\end{theorem}

\medskip\noindent
{\bf Remarks:}
(1). There is some gap between cases (i) and (ii), in particular, the
case $\ew k\asymp k^2$. In Theorem~\ref{th:manyswitch} we give more
general conditions under which (ii) holds, but we do not close this gap. \\
(2). It is easy to show (see, for example, Lemma \ref{lemma:staysclose})
that for any graph, if there are exceptional times of nonpercolation,
then the set of times $t\in[0,1]$ at which a fixed vertex percolates
is totally disconnected and hence has infinitely many connected components
with positive probability.

\medskip
{}From the proof of Theorem~\ref{th:dyntrans}.(i), it is easy to see that
for any graph, any edge dependent probabilities and any fixed vertex $x$,
if $I_n$ is the sum of the influences (see Section~\ref{sec:5} for the
definition of influence) for the event 
$\{x \mbox{ percolates to distance $n$ away}\}$,
then $\liminf_n I_n < \infty$ implies that the set of times $t\in[0,1]$ at
which $x$ percolates has finitely many connected components a.s.
Next, if $I_x(e)$ is the influence of the edge $e$ for the event
$\{\FCbetter(x,\infty)\}$,
it is easy to see from Fatou's lemma that
\begin{equation} \label{eq:Fatou}
\sum_e I_x(e) \le \liminf_n I_n.
\end{equation}
The next result tells us what we can conclude under the assumption that $\sum_e I_x(e) <\infty$.

\begin{theorem} \label{th:pivotal}
Consider dynamical percolation on any connected graph with possibly edge dependent
probabilities which percolates at a fixed time and let $x\in V$. Assume that
\begin{equation}
\label{e:piv}
\sum_e I_x(e) < \infty.
\end{equation}
Then a.s.\ $f(t):=\One_{\{\FCC(x,t,\infty)\}}$ is equal a.e.\ to a function of
bounded variation on $[0,1]$.
Moreover, this implies that there are no exceptional
times of nonpercolation.
\end{theorem}

\medskip\noindent
{\bf Remarks:} (1). Note that this result is applicable even in the
supercritical case. \\
(2). While it is easy to check that when the graph is a tree the summability above 
does not depend on $x$, interestingly, this is false in the general
context of connected graphs, even in the case of bounded degree.

\medskip
In~\cite{HPS}, it was argued that the events discussed in the above theorems
are measurable; a similar comment applies to all of our results.
Thus, measurability issues will not concern us here.

\medskip
As far as motivation, the questions that we look at give us a better
understanding of the stability properties of a critical infinite cluster
while at the same time fall into the general framework of studying
polar sets for stationary reversible Markov processes.

\medskip
The dynamical percolation results in \cite{HPS} were extended in
\cite{PS} and then further refined in \cite{K}. In \cite{SS}, it was shown
that there are exceptional times at criticality on the triangular lattice,
yielding the first example of a transitive graph with this property. 
We mention a few other papers where analogous dynamical sensitivity questions 
have been studied for other models. Analogous
questions for the Boolean model, where the points undergo independent
Brownian motions, were studied in \cite{BMW} and for certain interacting
particle lattice systems (where updates are therefore not done in an 
independent fashion) are studied in \cite{BrSt}. 
In \cite{BeSc}, it is shown that
there are exceptional two dimensional slices for the Boolean model in
four dimensions and finally, in \cite{JS}, dynamical versions of
Dvoretzky's circle covering problem are studied.

\medskip\noindent
{\bf Notation:}
(1). For subsets $A$ and $B$ of the vertices and $t$, we let
$\{\FCC(A,t,B)\}$ be the event that at time t there is an open path from
$A$ to $B$ and $\{\FCbetter(A,B)\}$
be the analogous event for ordinary percolation.
(If $B=\infty$, this has the obvious meaning.) 
In the context of trees with a distinguished root,
$\FCbetterTree(A,B)$ will mean that there is a path of open edges
connecting $A$ to $B$ along which the distance to the root is monotone increasing.
The notation $\FCCTree(A,t,B)$ is similarly defined.
\\
(2). We use $\asymp$ to denote the relationship between two quantities whose ratio is bounded away
from both 0 and $\infty$. \\
(3). $O(1)$ will denote a function bounded away from $\infty$,
$o(1)$ will denote a function approaching 0,
and $\Omega(1)$ will denote a function bounded away from 0. 

\medskip\noindent
{\bf  Convention:}
The edges are defined to be on at the times at which they
change state; in this way, the set of times an edge is on is a closed set.
As explained in \cite{HPS}, this modification is of no significance,
but allows some notational simplification in some topological arguments.

\medskip
The rest of the paper is organized as follows.
In Section~\ref{sec:2}, we prove Theorem~\ref{th:newexamples}.
In Section~\ref{sec:3}, we prove two lemmas
which will be needed for the proof of Theorem~\ref{th:sstrees1}.
We prove Theorem~\ref{th:sstrees1} in Section~\ref{sec:4},
Theorem~\ref{th:dyntrans} in Section~\ref{sec:5} and Theorem
\ref{th:pivotal} in Section~\ref{sec:6}. In Section~\ref{sec:01law},
we prove a certain 0-1 law for the evolution of the process
and finally we list some open questions in Section~\ref{sec:questions}.

\section{Two Examples}\label{sec:2}

\def\expectation{{\mathbb E}}
\def\proba{{\mathbb P}}
\def\integers{{\mathbb Z}}
\def\Lsquare{\mathbb Z^2}
\def\uproba{{P}}
\def\bfpsi{{\bf\Psi}}
\def\bfe{{\bf E}}
\def\cluster{{\cal{C}}}
\def\exponential{\mathrm{e}}
\def\d{\, \mathrm{d}}

The idea in the construction of the examples is rather simple; we
take the planar square lattice $\Lsquare$ and replace each edge by an appropriate
graph, with different graphs for different edges. For the example without
exceptional times, we will want the connection along the corresponding
graphs to be rather stable, while for the example with exceptional times,
we will want the connections to switch quickly.
The following lemma gives the existence of the necessary building blocks
for both examples. It contains a variant of Lemma 2.3 in \cite{HPS} with the crucial
difference being that the degrees are now bounded.

\begin{lemma}\label{l.switcheroo}
There is a sequence of finite graphs $G_j$ and pairs of vertices $x_j$ and $y_j$ 
in $G_j$, such that the following properties hold: 
\begin{enumerate} 
\item $\proba_{\frac{1}{2}}^{G_j}\bigl(\FCbetter(x_j,y_j)\bigr)>\frac{2}{3}$ 
for all $j$,
\item $\lim_{j\to\infty}\proba_{p}^{G_j}\bigl(\FCbetter(x_j,y_j)\bigr)=0$ for all $p<\frac{1}{2}$,
\item for every $\eps>0$ we have 
$$
\lim_{j\to\infty} \bfpsi^{G_j}_{\frac{1}{2}}\Bigl(
\bigcap_{t\in[0,\eps]}\{
\FCC(x_j,t,y_j)\}
\Bigr)=0\,,
$$
\item and there is some finite upper bound for the degrees of the vertices in $G_j$ (the
bound does not depend on $j$).
\end{enumerate}
\end{lemma}

\begin{proof}
 Let $H$ be obtained from the square grid in the plane
by replacing each edge by $m$ parallel edges, where $m$ is chosen
so that the probability that the origin percolates in $H$ at $p=1/2$ is at least $0.99$.
Let $v_i$ denote the vertex $(i,0)$ of $H$. Then for every $i$ we have
$\proba^{H}_{1/2}\bigl(\FCbetter(v_0,v_i)\bigr)\ge (0.99)^2>0.98$.
Hence, there is a finite subgraph $H_j$ of $H$ such that 
$\proba^{H_j}_{1/2}\bigl(\FCbetter(v_0,v)\bigr)\ge 0.98$ holds for every
$v\in A_j$, where $A_j:=\{v_i:1\le i\le 9\cdot 2^j\}$.
The graph $G_j$ is obtained by taking two disjoint copies of $H_j$
and connecting each of the vertices corresponding to $v_i\in A_j$ in
one copy to the vertex corresponding to $v_i$ in the other copy by a path
of length $j$, where the paths are of course disjoint.
The vertex $x_j$ is chosen as $v_0$ in one copy of $H_j$, while $y_j$ is
$v_0$  in the other copy. 
 The paths of length $j$ in $G_j$ connecting one copy of $H_j$
to the other will be called {\sl bridges}.

We now verify that $G_j$ satisfies the required properties. 
Let $B_j$ denote the set of vertices in $A_j$ connected to $v_0$ by an open path in $H_j$.
Since $\proba^{H_j}_{1/2}\bigl(\FCbetter(v_0,v)\bigr)\ge 0.98$
for all $v\in A_j$,
we have $\proba^{H_j}_{1/2}\bigl(|B_j|< (2/3)\,|A_j|\bigr)<0.9$.
This implies that in $G_j$ at $p=1/2$ with probability at least $(0.9)^2$
we have that the endpoints of at least $1/3$ of the bridges
are connected to $x_j$ within $x_j$'s copy of $H_j$ and to $y_j$ within
$y_j$'s copy of $H_j$. On this event, the conditional probability that $x_j$
and $y_j$ are not connected is at most 
$$(1-2^{-j})^{\frac{|A_j|}{3}}\le \exp(-2^{-j})^{\frac{|A_j|}{3}}= e^{-3}.$$
Thus, we get
$\proba_{1/2}^{G_j}(x_j\leftrightarrow y_j) \ge (0.9)^2\,(1-e^{-3})>2/3$,
proving 1.

If $p<1/2$, then the expected number of bridges that are open in $G_j$ is $|A_j|\,p^j=9\cdot 2^j\cdot p^j\to 0$
as $j\to\infty$, which proves 2.

In order to prove 3, fix some $\eps>0$, and consider dynamical percolation at $p=1/2$ on
$G_j$.
Let $t,s\in[0,\eps]$ satisfy $s\ne t$, and 
let $X^j_t$ denote the event that at time $t$ there is some bridge
in $G_j$ that is open.
Fix some ordering of the bridges in $G_j$, and let
$X^j_t(i)$ denote the event that the $i$'th bridge is open at time $t$.
Also let $\hat X^j_t(i)$ be the event that the $i$'th bridge is open 
at time $t$ and this does not hold for any smaller $i$.
Note that for every fixed $i$,
$$
\bfpsi_{1/2}^{G_j}\bigl({X^j_t\setminus X^j_t(i)\mid \,\hat X^j_s(i)}\bigl) \le \bfpsi_{1/2}^{G_j}\bigl(X^j_t\bigr).
$$
Therefore,
\begin{multline*}
\bfpsi_{\frac{1}{2}}^{G_j}\bigl({X^j_t,\,\hat X^j_s(i)}\bigl)=
\bfpsi_{\frac{1}{2}}^{G_j}\bigl({X^j_t\setminus X^j_t(i),\,\hat X^j_s(i)}\bigl)+
\bfpsi_{\frac{1}{2}}^{G_j}\bigl(X^j_t(i),\,\hat X^j_s(i)\bigl)
\\
\le
\bfpsi_{\frac{1}{2}}^{G_j}\bigl(X^j_t\bigr)
\bfpsi_{\frac{1}{2}}^{G_j}\bigl( \hat X^j_s(i)\bigl)
+
\bfpsi_{\frac{1}{2}}^{G_j}\bigl(X^j_t(i),\,\hat X^j_s(i)\bigl)\,.
\end{multline*}
On the other hand, the conditional probability of $X^j_t(i)$ given
$\hat X^j_s(i)$ does not depend on $i$ and goes to zero as $j\to\infty$ while $s\ne t$ are held fixed.
Thus,
$$
\bfpsi_{\frac{1}{2}}^{G_j}\bigl({X^j_t,\,\hat X^j_s(i)}\bigl)\le
\bfpsi_{\frac{1}{2}}^{G_j}\bigl(X^j_t\bigr)
\bfpsi_{\frac{1}{2}}^{G_j}\bigl( \hat X^j_s(i)\bigl) + o(1)\, 
\bfpsi_{\frac{1}{2}}^{G_j}\bigl(\hat X^j_s(i)\bigl)\,.
$$
As $X^j_s$ is the disjoint union of the events $\hat X^j_s(i)$, by summing
the above over $i$, we obtain
$$
\bfpsi_{\frac{1}{2}}^{G_j}\bigl({X^j_t,\,X^j_s}\bigl)\le
\bfpsi_{\frac{1}{2}}^{G_j}\bigl(X^j_t\bigr)
\bfpsi_{\frac{1}{2}}^{G_j}\bigl( X^j_s\bigl)+o(1)\,,
$$
as $j\to\infty$.

Set $X^j:=\int_{0}^\eps \One_{X^j_t}\,dt$. Fubini and the dominated convergence
theorem now imply that $\limsup_{j\to\infty}\Eb{(X^j)^2}-\Eb{X^j}^2\le 0$; that is,
the variance of $X^j$ tends to $0$. Since
$$
\Eb{X^j}=\eps\,\bfpsi_{\frac{1}{2}}^{G_j}({X^j_0})=\eps\,\bigl(1-(1-2^{-j})^{|A_j|}\bigr)
\underset {j\to\infty }\longrightarrow \eps\,(1-e^{-9})\,,
$$
 and the right hand side is smaller than $\eps$, it follows
that $\bfpsi_{1/2}^{G_j}({X^j=\eps})$ tends to $0$ as $j\to\infty$. This proves 3.

Claim 4 is obvious from the construction.
\QED
\end{proof}

\proofof{Theorem~\ref{th:newexamples}}
Both examples are obtained by replacing each edge $[x,y]$ in the square 
lattice $\Lsquare$ by a copy of some $G_j$, with $x_j$ identified
with $x$ and $y_j$ identified with $y$. The difference between the two 
examples has to do with the choice of $j$ for the different edges.

We start by proving (i). By property 1 of Lemma~\ref{l.switcheroo},
it follows that for each $j$ there is some positive integer $n_j>0$ such that
$$
\bfpsi^{G_j}_{\frac{1}{2}}\Bigl(\bigcap_{t\in[0,\frac{1}{n_j}]}\{\FCC(x_j,t,y_j)\}\Bigr)>
\frac{3}{5}\,.
$$
We may assume without loss of generality that the sequence $\{n_j\}$
is increasing in $j$.
We now define inductively an increasing sequence $\{R_j\}$.
Set $n_j^*:= n_{j+2}$.
For any two radii $0<r<r'$, let $\mathcal A(r,r')$ denote
the event that there is an open cycle in $\Lsquare$ separating
$\partial B(0,r)$ from $\partial B(0,r')$
where $\partial B(0,r):=\{x:|x|_\infty=r\}$ and $|x|_\infty$ denotes the $L_\infty$ 
norm of $x$. Let $R_0$ be so large that
$$
\proba_{\frac{3}{5}}(\FCbetter({B(0,R_0)},\infty)) \ge \frac{1}{2}\,.
$$
For all $j>0$, given $R_{j-1}$, we choose $R_j>R_{j-1}$ sufficiently large so that
$$
 \proba_{\frac{3}{5}}\bigl( 
    \FCbetter({B(0,R_j)}, \infty),\,\mathcal A(R_{j-1},R_j) \bigr)
 \ge 1- 2^{-j}\,(n_j^*)^{-1}.
$$
Let $G$ be obtained from $\Lsquare$ by replacing, for each $j>0$, each edge $e$
in the annulus $B(0,R_j)\setminus B(0,R_{j-1})$ by a new copy of $G_j$,
where $x_j$ and $y_j$ are identified with the endpoints of $e$.
By property 2 of the lemma, it follows that at every $p<1/2$, Bernoulli percolation
on $G$ a.s.\ has no infinite cluster. Hence $p_c(G)\ge 1/2$.

We now consider dynamical percolation on $G$ with parameter $p=\frac{1}{2}$,
and show that $\bfpsi^G_{1/2}$-a.s.\ there is an infinite percolation cluster at all times.
This, in particular, implies that $p_c(G)\le 1/2$; and hence $p_c(G)=1/2$.

For $I\subseteq[0,\infty)$, let $\mathcal A_j(I)$ denote the event that at all times 
$t\in I$ there
is an open cycle in $G$ separating $\partial B(0,R_j)$ from $\partial B(0,R_{j-1})$ and an
open path in $G$ connecting $\partial B(0,R_{j-1})$ with $\partial B(0,R_{j+1})$.
Then $\bfpsi^G_{1/2} \{\mathcal A_j([0,\, 1/n_{j+1}])\} \ge 1-2^{-j+2}/n^*_{j-1}$,
whence $\bfpsi^G_{1/2} (\mathcal A_j([0,\, 1])) \ge 1-2^{-j+2}$.
Now note that if $\bigcap_{j>k} \mathcal A_j([0,\,1])$ holds
for some $k$, then there is percolation in $G$ for every $t\in[0,1]$.
Since $\bfpsi^G_{1/2}\Bigl( \bigcap_{j>k} \mathcal A_j([0,\,1])\Bigr)\ge 1-2^{-k+2}$,
this gives $\bfpsi^G_{1/2} \Bigl(\bigcap_{t\in[ 0,1]} \cluster_t\Bigr) =1$,
which implies (i).

\medskip

We now turn to the proof of (ii). 
Using Lemma~\ref{l.switcheroo} together with the proof of
the second part of Theorem 1.2 in \cite{HPS}, it is easily seen
that if we replace the $i$th edge by $G_{j_i}$ with the sequence
$\{j_i\}$ growing to infinity sufficiently fast, we obtain an example of the
desired form.
\QED

\section{Some lemmas}\label{sec:3}

We now consider a spherically symmetric tree with spherically symmetric edge
probabilities. As in the introduction, $W_n$ will denote the number of 
vertices in $T_n$ that are connected to the root, and $w_n$ denotes the
expectation of $W_n$.

By Theorem 2.3 of~\cite{Ly} (together with the proof of
Theorem 2.4 in that paper and the fact that
for a spherically symmetric kernel, the measure that minimizes
energy is the uniform measure, a fact which in turn is obtained using
convexity of energy together with symmetry), it follows that
\begin{equation} \label{e:Lyonsagain}
\frac{\ew n^2}{E[W_n^2]} \le P(W_n >0)\le \frac{2\ew n^2}{E[W_n^2]}.
\end{equation}
The second inequality yields 
\begin{equation}\label{e:L2bounded}
E[W_n^2|W_n>0]\le 2 E[W_n|W_n>0]^2,
\end{equation}
which will be useful below.

\begin{lemma} \label{lemma:array}
Consider an indexed collection $\{X_{i,j}\}_{i\ge 1, 1\le j \le N_i}$
of nonnegative mean 1 random variables such that (1) for each $i$,
$\{X_{i,j}\}_{1\le j \le N_i}$ are i.i.d. and (2) the entire family of
random variables is uniformly integrable. Then for each $\epsilon > 0$,
there is $c>0$ such that for each $i$,
$$
P\left(\sum_{j=1}^{N_i} X_{i,j}\le N_i(1-\epsilon)\right) \le e^{-c N_i}.
$$
\end{lemma}

\proof
Let
$\epsilon > 0$. By uniform integrability, there exists $h=h(\epsilon)$ such
that for all $i$ and $j$,
$$
E(X_{i,j}\wedge h) \ge 1-\frac{\epsilon}{2}.
$$
We then have
\begin{multline*}
P\left(\sum_{j=1}^{N_i} X_{i,j}\le N_i(1-\epsilon)\right) \le
P\left(\sum_{j=1}^{N_i} X_{i,j}\wedge h \le N_i(1-\epsilon)\right)
\\\le
P\Bigl(\sum_{j=1}^{N_i} X_{i,j}\wedge h \le
N_i\bigl(E(X_{i,j}\wedge h)-\frac{\epsilon}{2}\bigr)\Bigr).
\end{multline*}
As we now have bounded random variables, the standard
Chernoff bound arguments allow us to bound the latter
by $e^{-c N_i}$ for some fixed $c=c(\epsilon,h)>0$. \QED

\begin{lemma} \label{lemma:staysclose}
Fix a connected graph $G$ and $x\in V(G)$. Let
$B_M:=\{y:d_G(x,y)\le M\}$ where $d_G$ is the graph distance.
Then the following are equivalent. \\
(i).
$$
\bPsi_p(\, \calC_t \, \mbox{ occurs for every  } \, t \, )=1.
$$
(ii).
$$
P(\exists M: \FCC(B_M,t,\infty)\,\,\forall t\in [0,1])=1.
$$
(iii).
$$
P(\FCC(x,t,\infty)\,\,\forall t\in [0,1])> 0.
$$
\end{lemma}

\proof
The implication
(iii) $\Rightarrow$ (i) is immediate from Kolmogorov's 0-1 Law.
The implication
(ii) $\Rightarrow$ (iii) is easy and left to the reader.
We now show that (i) implies (ii).
If (ii) is false, Kolmogorov's 0-1 Law
implies that the event in (ii) has probability 0. 
Positive association of the process
and the above 0-1 Law then would yield that for all $\delta > 0$,
\begin{equation} \label{e:delta}
P(\exists M: \FCC(B_M,t,\infty)\; \forall t\in [0,\delta])=0\,.
\end{equation}

Now, for each vertex $v$, let $U_v$ be the open set of times in
$[0,1]$ at which $v$ is not percolating.
Countable additivity
and~(\ref{e:delta})
easily imply that a.s.\ each $U_v$ is dense. The Baire Category
Theorem implies that a.s.
$$
\bigcap_{v}U_v \neq \emptyset\,.
$$
However, this intersection is exactly the set of nonpercolating times
in $[0,1]$ and hence (i) is false. \QED

\medskip\noindent
{\bf Remarks:}  Observe that given any graph which percolates at criticality
and for which there are exceptional nonpercolating times, using the
$U_v$'s as above, the Baire Category Theorem gives that
the set of nonpercolating times in $[0,1]$ is a dense $G_\delta$ set of zero measure.
An additional use of the Baire Category Theorem tells us that if we hook up
a finite number of such graphs at a common vertex,
there will still be nonpercolating times and they will also form a dense
$G_\delta$ of zero measure. This situation is very different
from the case where one looks at time sets corresponding to the times
at which a tree, which does not percolate at criticality
(in static percolation), percolates;
such time sets do not necessarily intersect each other.

\section{Proof of Theorem \ref{th:sstrees1}}\label{sec:4}

We now begin with the
\proofof{Theorem~\ref{th:sstrees1}(i)}
Recall that $\rho$ denotes the root of the tree.
Fix an $\alpha > 2$,
and assume that $\lim_n \frac{\ew n}{n(\log n)^{\alpha}}=\infty$.
Choose $\epsilon >0$ such that
$2 +2\,\epsilon <\alpha$. Let $n_k:=2^{2^k}$. (So $n_0=2$ and $n_{k+1}=n^2_{k}$.)
For each $k$ and each $i\in \{1,\dots,n_k^2\}$, let $I^k_i=[(i-1)/n_k^2,i/n_k^2]$.
Let $A_i^k:=\{x\in T_{n_k}: \FCC(\rho,t,x)\;\forall t\in I^k_i\}$,
and let $G_k$ denote the event that $|A_i^k| \ge \ew{ n_k}/(\log n_k)^{\eps}$
holds for every $i\in\{1,2,\dots,n_k^2\}$.
We need to obtain a good bound on $P(G^c_{k+1}|{\mathcal F}_{n_k})$ on the event $G_k$,
where ${\mathcal F}_n$ is the $\sigma$-algebra generated by the evolution
of the first $n$ levels of the tree. 
The key proposition, whose proof we give afterwards, is the following.

\begin{proposition} \label{pr:key}
There exists $\gamma >1$ so that for all large $k$,
if $A\subseteq T_{n_k}$ is fixed with
$|A|\ge \ew {n_k}/(\log n_k)^\epsilon$, then 
$$
P\Bigl(\bigl|\{x\in T_{n_{k+1}}:
\FCCTree(A,t,x)\,\, \forall t \in I^{k+1}_1\}\bigr|
\le \ew {n_{k+1}}/(\log n_{k+1})^\epsilon\Bigr)
\le e^{-(\log n_k)^\gamma}.
$$
\end{proposition}

\medskip\noindent
We now first complete the proof of Theorem~\ref{th:sstrees1}(i) by noting
that it is easy to see that Proposition \ref{pr:key} implies that
for large $k$, we have that on $G_k$
$$
P(G_{k+1}^c|{\mathcal F}_{n_k})\le n^2_{k+1} e^{-(\log n_k)^\gamma}.
$$
Since $\gamma>1$, we have
$$
\sum_k
 n^2_{k+1} e^{-(\log n_k)^\gamma}<\infty\,.
$$
For any finite $k'$, we have $P\Bigl(\bigcap_{k\le k'} G_k\Bigr)>0$.
Hence, the above implies that $P(G_k\; \forall k)>0$,
and since $\bigcap_k G_k\subseteq \{\FCC(\rho,t,\infty)\;\forall t\in[0,1]\}$, this implies
$$
P(\FCC(\rho,t,\infty) \,\,\forall t\in [0,1])> 0.
$$
This yields the required result by Lemma \ref{lemma:staysclose}.
\QED

Before starting the proof of Proposition \ref{pr:key},
we first need the following lemma.

\begin{lemma} \label{lemma:Lyonsnlogn}
Consider a spherically symmetric tree with spherically symmetric edge probabilities,
and assume that for some $\beta> 1$,
$\ew n\ge \Omega(1)\, n\,(\log n)^\beta$
holds for every $n$. If $x\in T_{n_k}$, then
$$
P(\FCbetterTree(x,T_{n_{k+1}}))\, \ew {n_k}\ge
\Omega(1)\, (\log n_k)^{\beta-1}.
$$
\end{lemma}

\proof
It is easy to see that for $x\in T_{n_k}$,
the expected number of vertices in $T_\ell$ connected to $x$ within $T^x$
is $\ew {\ell}/\ew {n_k}$ for $\ell\ge n_k$. Hence by (\ref{e:lyonsstrong}),
if $x\in T_{n_k}$, we have that
$$
P(\FCbetterTree(x,T_{n_{k+1}}))\asymp
\Bigl(\sum_{\ell=n_k+1}^{n_{k+1}} \frac{\ew{n_k}}{\ew{\ell}}\Bigr)^{-1}
\ge \Omega(1)\,  \frac{1}{\ew {n_k}}\,
\Bigl(\sum_{\ell=n_k+1}^{n_{k+1}}
\frac{1}{\ell(\log \ell)^{\beta}} \Bigr)^{-1}.
$$
Next
$$
\sum_{\ell=n_k+1}^{n_{k+1}} \frac{1}{\ell(\log \ell)^{\beta}}
\asymp  \int_{n_k}^{n_{k+1}} \frac{1}{x(\log x)^{\beta}} \,dx
=
\int_{\log n_k}^{\log(n_{k+1})} \frac{1}{u^{\beta}}\, du
\asymp  (\log n_k)^{1-\beta},
$$
since $n_k=2^{2^k}$,
completing the proof.
\QED

\proofof{Proposition~\ref{pr:key}}

For $x \in T_{n_k}$, let $R_x$ be the number of vertices at level $n_{k+1}$
which are connected to $x$ within $T^x$ throughout $[0,1/n^2_{k+1}]$ and
let $R_k$ denote a random variable which has distribution $R_x$.
The expected number of vertices at level $n_{k+1}$ which are connected to $x$
within $T^x$ at time 0 is $\ew{n_{k+1}}/\ew {n_k}$.
Since a given path of length $n_{k+1}-n_k$ is updated during $[0,1/n^2_{k+1}]$
with probability $o(1)$, we have
\begin{equation} \label{eq:needlaterfirst}
E[R_k] = \frac{\ew{n_{k+1}}}{\ew {n_k}} \,\bigl(1- o(1)\bigr),
\end{equation}
as $k\to\infty$.

\begin{lemma} \label{lemma:2ndmomentbound}
Let $\tilde{R}_k$ have distribution $R_k$ conditioned on $\{R_k >0\}$.
Then
$$
E[(\tilde{R}_k)^2]\le O(1)E[(\tilde{R}_k)]^2.
$$
\end{lemma}

\proof
Fix some $x\in T_{n_k}$, and let $R'_x:=\bigl|\{y\in T_{n_{k+1}}:\FCCTree(x,0,y)\}\bigr|$.
We have argued above that $E[R_k]\ge \bigl(1-o(1)\bigr)E[R'_x]$.
This implies $E[R_k]\asymp E[R'_x]$.
A similar argument gives $P(R'_x>0)\asymp P(R_k>0)$.
Since $R'_x\ge R_x$, this together with~\eqref{e:L2bounded} easily leads to the statement;
the details are left to the reader.
\QED

\begin{lemma} \label{lemma:enoughguys}
There exists $\gamma >1$ so that for all $\delta>0$, we have
that for large $k$, if $A\subseteq T_{n_k}$ with
$|A|\ge\ew {n_k}/(\log n_k)^\epsilon$, then
$$
P\Bigl(\bigl|\{x\in A: R_x >0\}\bigr| \le
(1-\delta)\,\frac{P(R_k>0)\,\ew {n_k}}{(\log n_k)^\epsilon}\Bigr)
\le e^{-(\log n_k)^\gamma}.
$$
\end{lemma}

\proof
The random variable $X:=\bigl|\{x\in A: R_x >0\}\bigr|$  has a binomial distribution 
with parameters $|A|$ and $P(R_k >0)$. The probability in the statement of the lemma 
is at most
$$
P\bigl(X\le E[X](1-\delta)\bigr).
$$
By standard large deviations (see for example Corollary A.1.14 in \cite{AS}), the latter is a most
$2\,e^{-c_\delta E(X)}$.  Lemma \ref{lemma:Lyonsnlogn} 
and our choice of $\eps$ imply that
$E[X]\ge \Omega(1)\, (\log n_k)^{1+\epsilon}$, proving the claim.
\QED

\begin{lemma} \label{lemma:secondtermbound}
There exists $\delta >0$ and $\gamma >1$ such that for all large $k$, if
$$
M\ge (1-\delta)\,\frac{P(R_k>0)\,\ew {n_k}}{(\log n_k)^\epsilon}
$$
and $Y_1,\ldots,Y_{M}$ are i.i.d.\
with the distribution of $\tilde{R}_k$ (defined
in Lemma \ref{lemma:2ndmomentbound}), then
\begin{equation}
\label{e.2tb}
P\left(\sum_{i=1}^{M} Y_i \le
\frac{\ew{n_{k+1}}}{(\log n_{k+1})^\epsilon}\right)\le e^{-(\log n_k)^\gamma}.
\end{equation}
\end{lemma}

\proof
Choose $\delta$ so that
\begin{equation} \label{eq:delta}
\frac{1}{2^\epsilon(1-\delta)} <1.
\end{equation}
Our lower bound on $M$ and an 
easy calculation shows that the left hand side of~\eqref{e.2tb} is bounded by
$$
P\left(\frac{1}{M}\sum_{i=1}^{M}\frac{Y_i}{E[Y_i]}
\le
S_k\right),\qquad\text{where }S_k:=
\frac{\ew{n_{k+1}} (\log n_k)^\epsilon}
{\ew {n_k} (\log n_{k+1})^\epsilon(1-\delta)E[R_k]}\,.
$$
The expression~\eqref{eq:needlaterfirst}
for $E[R_k]$ implies that $\lim_{k\to\infty} S_k=1/(2^\epsilon(1-\delta))$.
Since a family of random variables which have a uniform bound on
their second moments is uniformly integrable, Lemmas \ref{lemma:array}
and \ref{lemma:2ndmomentbound} and (\ref{eq:delta})
imply that
$$
 P\Bigl(\frac{1}{M}\sum_{i=1}^{M}\frac{Y_i}{E[Y_i]}
\le
S_k\Bigr)\le e^{-cM},
$$
 for some $c>0$ and all large $k$.
Lemma \ref{lemma:Lyonsnlogn} insures that
$M\ge \Omega(1)(\log n_k)^{1+\epsilon}$,
completing the proof.
\QED

One finally notes that Proposition \ref{pr:key}
is a consequence of Lemmas \ref{lemma:enoughguys} and
\ref{lemma:secondtermbound}. \QED

\medskip\noindent
{\bf Remark:} In the proof of Theorem~\ref{th:sstrees1}(ii),
we separate things into the two cases $\alpha<2$ and $\alpha=2$
but we emphasize that this is done for presentational purposes only.

\medskip\noindent
We now move to
\proofof{Theorem~\ref{th:sstrees1}(ii); case $\alpha<2$}
Let $A:=\{\FCC(\rho,t,\infty) \,\,\forall t\in [0,1]\}$.
By Lemma \ref{lemma:staysclose}, it suffices to show that $P(A)=0$
and for this it suffices to show that for every $M>0$, there is an event
$G=G(M)$ so that $P(G) \ge 1- 2/M$ and $P(A|G)=0$. We now fix such an $M$.
The $O(1)$ terms appearing below may (and will) depend on $M$ (but they
will of course be independent of the level of the tree under discussion).

For the moment, we consider our percolation at a fixed time.
It is well known that $\{W_n/\ew n\}$ (recall $W_n$ is the number
of vertices on the $n$'th level connected to the root)
is a nonnegative martingale and hence converges a.s.\ to a random variable
denoted
$W_\infty$ with $E[W_\infty]\le 1$.
Doob's inequality tells us that
\begin{equation} \label{eq:2overM}
P\left(\frac{W_n}{\ew n}\ge M \mbox{ for some } n\ge 0\right)
\le \frac{1}{M}\,.
\end{equation}

Returning to our dynamical model, we let $W_{n,t}$ be the analogue of
$W_n$ above but at time $t$. We now define
$$
G:= \left\{ \mu\bigl\{t\in [0,1]:
{W_{n,t}}/{\ew n}\ge M \mbox{ for some } n\ge 0\bigr\}<
\frac{1}{2}\right\},
$$
where $\mu$ denotes Lebesgue measure.
Fubini's theorem, Markov's inequality and (\ref{eq:2overM})
easily yield that $P(G)\ge 1- 2/M$.  We will
show that $P(A|G)=0$, completing the proof.

Set $m_n:= \lfloor M\,\ew n\rfloor$.
For all $B\subseteq T_n$ with $|B|\le m_n$,
let $\tilde{B}$ be a subset of $T_n$ containing
$B$ such that
$|\tilde{B}|= m_n$,
and such that $\tilde B$ is a deterministic function of $B$.
Of course, this can only be done for
$n\ge N=N(M):=\min\bigl\{k:|T_k|\ge m_k \bigr\}$.
If $|B|>m_n$, we take $\tilde B$ to be the leftmost $m_n$
elements of $B$.

Let $S_{n,t}$ be the set of vertices in $T_n$ that are
connected to $\rho$ by open paths at time $t$.
Then $W_{n,t}=|S_{n,t}|$.
For each $n\ge N=N(M)$, define the random variable
$$
X_n:=
\mu\bigl\{t\in[0,1]:
W_{n,t}\le m_n,\,
\FCCnotTree(\tilde{S}_{n,t},t,\infty)\bigr\}\,.
$$

The key step is to carry out a conditional second moment
argument on $X_n$ conditioned on the evolution of the first
$n$ levels 
on that part of the probability space where
something ``good'' happens. The following proposition
 will be the consequence of this conditional second moment
argument.

\begin{proposition} \label{pr:keysecond}
There exists $c=c(M)>0$ such that for all $n$ sufficiently large
$$
P\left(X_n > 0|{\mathcal F}_n\right) \ge c \, \mbox{ on } G
$$
where ${\mathcal F}_n$ is the $\sigma$-algebra generated by the evolution
of the first $n$ levels of the tree.  
\end{proposition}

\medskip\noindent
We postpone the proof of the proposition, and continue with the proof
of the theorem.
It is clear that $\{X_n > 0\}\subseteq A^c$ and hence
$$
P(A^c|{\mathcal F}_n) \ge c \, \mbox{ on } G.
$$
Letting $n\to\infty$, Levy's 0-1 Law implies
that the left hand side approaches $\One_{A^c}$ a.s. As $c>0$,
we conclude that $P(A|G)=0$, as desired. \QED

Before starting the proof of Proposition \ref{pr:keysecond}, we need
a lemma. Let
$$
q_n:=P\left(\FCCTree(x,0,\infty)\right) \mbox{ and }
q_n(t):=P\left(\{\FCCTree(x,t,\infty)\}\cap\{\FCCTree(x,0,\infty)\}\right),
$$
where $x\in T_n$.

It is easy to check that the proof of Lemma \ref{lemma:Lyonsnlogn} shows that
\begin{equation} \label{e:q_n}
q_n \asymp  \frac{1}{n\log n}\,.
\end{equation}

\begin{lemma} \label{lemma:correlation}
$$
q_n(t)\le \frac{O(1)q_n^2}{t}\,.
$$
\end{lemma}

\proof
Fix $x\in T_n$ and $t\in(0,1]$.
Suppose that $\FCCTree(x,0,\infty)$,
and condition on the left most open path
$\pi=(\pi_0,\pi_1,\dots)$ from
$x$ to $\infty$ inside $T_x$ at time $0$.
Let $K_j$ be the event that at time $t$ there is
an open path from $x$ to $\infty$ that shares exactly
$j$ edges with $\pi$.
Because in the complement of $\pi$ the conditional law
of the dynamical percolation is dominated by
the unconditional law, we clearly have
\begin{multline*}
P\left(K_j\:|\:\FCCTree(x,0,\infty)\right)\le 
P\left(\FCCTree(\pi_j,t,\infty)\right)\,
P\left(\FCCTree(x,t,\pi_j)\:|\:\FCCTree(x,0,\infty)\right)
\\
=
q_{n+j}
\prod_{i=1}^{j} \Bigl(p_{n+i}(1-e^{-t})+e^{-t}\Bigr).
\end{multline*}
Since $P(K_\infty)=0$, we get
\begin{multline*}
q_n(t)
 =q_n\,
P\left(\FCCTree(x,t,\infty)\:|\:\FCCTree(x,0,\infty)\right)
\le
q_n\,
\sum_{j=0}^\infty
P\left(K_j\:|\:\FCCTree(x,0,\infty)\right)
\\
\le
q_n\,
\sum_{j=0}^\infty
q_{n+j}
\prod_{i=1}^{j} \Bigl(p_{n+i}(1-e^{-t})+e^{-t}\Bigr).
\end{multline*}
As the $p_i$'s are bounded away from $1$,
there exists a constant $\epsilon_0 \in(0,1)$
 such that each factor in the product
on the right is at most
$1-\epsilon_0 t$ (regardless of the choice of $t$ in $(0,1]$). 
Hence, the above gives
\begin{multline*}
q_n(t)\le q_n\,\sum_{j=0}^\infty
q_{n+j} \, (1-\eps_0\,t)^j
\le q_n\,\sup\{q_{n+j}:j=0,1,\dots\}\,\sum_{j=0}^\infty(1-\eps_0\,t)^j
\\
=
q_n\,\sup\{q_{n+j}:j=0,1,\dots\}\,(\eps_0\,t)^{-1}.
\end{multline*}
Now an appeal to~\eqref{e:q_n}
completes the proof.
\QED

Let
\begin{equation} \label{e:tilde1}
\tilde{q}_n:=1-q_n.
\end{equation}
Next, letting $\tilde{q}_n(t)$ be the probability that a given vertex
at level $n$ does not percolate to $\infty$
both at time 0 and at time $t$, we easily have that
\begin{equation} \label{e:tilde2}
\tilde{q}_n(t)=1-2q_n+q_n(t).
\end{equation}

We use~\eqref{e:tilde1} and~\eqref{e:tilde2},
to obtain
$$
\frac{\tilde{q}_n(t)}{\tilde{q}^2_n}
=
\frac{1-2\,q_n+q_n(t)}{(1-q_n)^2}
=1+\frac{q_n(t)-q_n^2}{(1-q_n)^2}
\le
1+\frac{q_n(t)}{(1-q_n)^2}
\,.
$$
By Lemma~\ref{lemma:correlation} and~\eqref{e:q_n} we therefore get 
\begin{equation}
\label{e:boundonbtilde}
\frac{\tilde{q}_n(t)}{\tilde{q}^2_n}
\le 1+ O\bigl(q^2_n/t\bigr).
\end{equation}

We can now carry out the

\proofof{Proposition~\ref{pr:keysecond}}

We apply a conditional second moment argument.
First, it is immediate that for any $n\ge N$
$$
E[X_n|{\mathcal F}_n]\ge \frac{1}{2}
(\tilde{q}_n)^{m_n}\,\One_G\,.
$$
In order to estimate $ E[X_n^2|{\mathcal F}_n]$, we note that
$$
\PB{\FCCnotTree(\tilde S_{n,s},s,\infty),
\FCCnotTree(\tilde S_{n,t},t,\infty)\md \mathcal F_n}
=
{\tilde q}_n(|t-s|)^{|\tilde S_{n,s}\cap\tilde S_{n,t}|}\,
{\tilde q}_n^{|\tilde S_{n,s}\setminus\tilde S_{n,t}|+|\tilde S_{n,t}\setminus\tilde S_{n,s}|}.
$$
Since ${\tilde q}_n(t)\ge {\tilde q}_n^2$, this gives for every $n\ge N$ a.s.
\begin{equation} \label{e:2ndmombound}
E[X_n^2|{\mathcal F}_n]
\le \int_0^1\int_0^1 
\tilde{q}_n(|t-s|)^{m_n}\,dt\,ds
\le 2
\int_0^1 \tilde q_n(t)^{m_n}\,dt\,.
\end{equation}
 Using the trivial bound
${\tilde{q}_n(t)}\le {\tilde{q}_n}$
for $t\le 1/n$ and the bound~\eqref{e:boundonbtilde} for larger
values of $t$, we get that on $G$
\begin{equation} \label{e:ratiobound}
\frac{E[X_n^2|{\mathcal F}_n]}{E[X_n|{\mathcal F}_n]^2}\le 8
\int_0^{{\frac{1}{n}}}
\left(\frac{1}{\tilde{q}_n}\right)^{m_n}dt
+
8\int_{\frac{1}{n}}^1
\Bigl(1+ O\bigl(q^2_n/t\bigr)\Bigr)
^{m_n}dt\,.
\end{equation}
Using (\ref{e:q_n}) and~\eqref{e:tilde1},
if $\alpha <2$, then the first integrand is easily checked to be at most
$O(1)\,n^\sigma$ for some $\sigma< 1$ (and in fact for any $\sigma< 1$
with the $O(1)$ term then of course depending on $\sigma$)
and hence the first integral
goes to 0. If $\alpha \le 2$, then, using (\ref{e:q_n}),
it is easy to check that the
second integrand, when $t\ge \frac{1}{n}$, is at most $O(1)$.
So the ratio of the conditional second moment and
the conditional first moment squared on $G$ is bounded above and so
the (conditional) Cauchy Schwartz inequality yields the claim of the
proposition.
\QED

\proofof{Theorem~\ref{th:sstrees1}(ii); case $\alpha =2$}
For any integers $n\ge L\ge 1$, and any $v\in T_L$, let
$W^{v}_n$ be the number of vertices at level $n$
connected to $\rho$ which are in $T^v$.

\begin{lemma}
\label{l:good}
Letting
$E_{L,\epsilon}:=\{W^{v}_n \le \epsilon\ew n \,\,\forall n\ge L,
\,\,\forall v\in T_L\}$, we have that for all $\epsilon >0$,
$$
\lim_{L\to\infty} P(E_{L,\epsilon})=1.
$$
\end{lemma}

\proof
Fix $\epsilon >0$ and $v\in T_L$. 
Since $W^v_n/\Eb{W^v_n}$ is a martingale with respect to $n$
(for $n\ge L$), we have
\begin{equation}\label{e.L1}
\begin{aligned}
P(W^{v}_n \ge \epsilon\ew n \mbox{ for some } n\ge L)
&
=
P\bigl(W^{v}_n \ge \epsilon\, E[W^{v}_n]\,|T_L| \mbox{ for some } n\ge L\bigr)
\\ &
\le \frac{1}{\epsilon^2 |T_L|^2}\,
\sup_{n\ge L}\frac{E[(W^{v}_n)^2]}{E[W^{v}_n]^2}\,,
\end{aligned}
\end{equation}
by Doobs $L_2$ martingale inequality.
The estimate (\ref{e:Lyonsagain}) gives for $n\ge L$
\begin{equation}
\label{e.L2}
\frac{E[(W^{v}_n)^2]}{E[W^{v}_n]^2}
\le \frac{O(1)}{P(W^v_n>0)}
\le \frac{O(1)}{P(\FCbetter(\rho,v))\,q_L}
= \frac{O(|T_L|)}{\ew L\,q_L}\,.
\end{equation}
We sum~\eqref{e.L1} over $v\in T_L$ and use~\eqref{e.L2}
as well as
(\ref{e:q_n}), to obtain
$$
P(E^c_{L,\epsilon})\le \frac{O(1)\,L\,\log L}{\ew L\, \epsilon^2}
$$
which approaches 0 as $L\to\infty$, since $\alpha>1$.
\QED

Next, using $\ew n \asymp n(\log n)^2$ and (\ref{e:q_n}),
choose an $\epsilon >0$ sufficiently small so that
$({1}/{\tilde{q}_n})^{\epsilon \ew n-1}\le n$
for all $n$ sufficiently large,
and set $m_n:=\lfloor \epsilon\,\ew n\rfloor$.
Let $E_{L,\epsilon,t}$ denote the event that $E_{L,\epsilon}$
occurs at time $t$, let
${\mathcal G}_{L,\epsilon}:=\{t\in[0,1]:E_{L,\epsilon,t}\}$
and let
$\tmG_{L,\eps}$ be the (closed) support of the restriction
of the Lebesgue measure $\mu$ to $\mathcal G_{L,\eps}$.
Finally, let
$G_{L,\epsilon}:=\{\tmG_{L,\eps}\ne \emptyset\}=
\{\mu(\mathcal G_{L,\eps})\ne 0\}$.
Lemma~\ref{l:good} easily implies that $\lim_{L\to\infty}P(G_{L,\epsilon})=1$.

For any vertex $v$, let
$$
{\mathcal T}^v:=
\{t\in[0,1]:\FCCnot(\rho,t,v)\}
\cup\{t\in[0,1]:\FCCnotTree(v,t,\infty)\}\,,
$$
which is the set of times in $[0,1]$ in which $\rho$ does not
connect to $\infty$ through $v$.
Note that ${\mathcal T}^v$ is open.

\begin{proposition} \label{pr:dense}
With the above choice of $\eps>0$, for all $L$ and $v\in T_L$,
$$
P({\mathcal T}^v \cap \tmG_{L,\epsilon}
\mbox{ is dense in } \tmG_{L,\epsilon})=1.
$$
\end{proposition}

\noindent
Given this proposition, the Baire category theorem
(or an easy induction) yields that
$$
P\Bigl(
\tmG_{L,\epsilon}\cap\bigcap_{v\in T_L} \mathcal T^v
\mbox{ is dense in }\tmG_{L,\epsilon}\Bigr)=1
$$
and hence
$$
P(A^c|G_{L,\epsilon})=1.
$$
Since $\lim_{L\to\infty}P(G_{L,\epsilon})=1$, we are done. \QED

\proofof{Proposition~\ref{pr:dense}}
Fix $L$ and $v\in T_L$.
By countable additivity, it suffices to show that
for all open intervals $I$ with rational endpoints,
\begin{equation} \label{e:Ienough}
P\bigl(\mu(I\cap \mG_{L,\eps})=0\text{ or }
\mu(\mathcal T^v\cap I\cap \mG_{L,\eps})>0\bigr)=1\,.
\end{equation}
Set $Y:=\mu(I\cap \mG_{L,\eps})$
and $Y_n:=\Eb{Y\md \mathcal F_n}$.
We claim that for some constant $c>0$, depending
only on $I$ and $L$, and for all sufficiently large $n$, we have
\begin{equation}
\label{e.claim}
P\bigl(
\mu(\mathcal T^v\cap I\cap\mG_{L,\eps})>0
\mid\mathcal F_n\bigr)\ge c\,Y_n^2\,.
\end{equation}
Clearly, $Y_n\to Y$ a.s.,
while Levy's 0-1 Law implies that the left hand side
converges a.s.\ to $\One_{\{\mu(\mathcal T^v\cap I\cap \mG_{L,\eps})>0\}}$.
Therefore,~\eqref{e.claim} implies~\eqref{e:Ienough} and the proposition.

For all $B\subseteq T_n\cap T^v$ with $|B|\le m_n$,
let $\tilde{B}$ be a subset of $T_n\cap T^v$ containing $B$
such that $|\tilde{B}|= m_n$ and $\tilde B$ is a deterministic function
of $B$.
(This only works for large enough $n$ so that
$|T^v\cap T_n|\ge m_n$.)
If $|B|>m_n$, let $\tilde B$ be the subset of $B$
consisting of the leftmost $m_n$ elements of $B$.
Let $S_{n,t}^v$ denote the set of vertices in $T^v\cap T_n$ that
are connected to $\rho$ at time $t$, and 
define
$$
X_n:=\mu\Bigl(\bigl\{t\in I\cap \mG_{L,\eps}:
\FCCnotTree(\tilde S^v_{n,t},t,\infty)\bigr\}\Bigr)\,.
$$
Then
$$
E[X_n\mid{\mathcal F}_n]=
\int_I P(t \in \mG_{L,\epsilon}\mid {\mathcal F}_n)\,
P(\FCCnotTree(\tilde{S}^{v}_{n,t},t,\infty)
\mid t \in \mG_{L,\epsilon},\,{\mathcal F}_n)\,dt.
$$
Since our process is positively associated even when conditioned on
${\mathcal F}_n$,
the second factor in the integrand is at least as large as
$
P(\FCCnotTree(\tilde{S}^{v}_{n,t},t,\infty\mid {\mathcal F}_n))
=
(\tilde q_n)^{m_n},
$
and hence
the above gives
$$
E[X_n\mid{\mathcal F}_n]\ge Y_n\,(\tilde{q}_n)^{m_n}.
$$

For the conditional second moment, let
$$
X_n^*:= \mu\Bigl(\bigl\{t\in I:
\FCCnotTree(\tilde S^v_{n,t},t,\infty)\bigr\}\Bigr)\,.
$$
Then $X_n^*\ge X_n$. Arguing as in the case $\alpha<2$, we get
$$
\Eb{X_n^2\md \mathcal F_n}\le\Eb{(X_n^*)^2\md \mathcal F_n}
\le 2\, \mu(I)\,\int_0^{\mu(I)} \tilde q_n(t)^{m_n}\,dt\,.
$$
We take $n$ larger than $1/\mu(I)$, and use the bounds
$\tilde q_n(t)\le \tilde q_n$ and~\eqref{e:boundonbtilde},
to get
$$
\frac{E[X_n^2|{\mathcal F}_n]}{E[X_n|{\mathcal F}_n]^2}
\le
\frac{2\,\mu(I)}{Y_n^2}\int_0^{1/n} (\tilde q_n)^{-m_n}\,dt
+
\frac{2\,\mu(I)}{Y_n^2}\int_{1/n}^{\mu(I)}
\bigl(1+ O( {q^2_n}/{t})\bigr)
^{m_n}\,dt\,.
$$
By our choice of $\eps$ and $m_n$, the left integral is bounded.
As we have seen in the previous case, the integrand of the right integral is
also bounded. 
The (conditional) Cauchy Schwartz inequality therefore gives~\eqref{e.claim}.
\QED

\section{Proof of Theorem \ref{th:dyntrans}} \label{sec:5}

We first recall the definitions of pivotality and influence.

\medskip\noindent
{\bf Definition:}
An edge $e$ is {\sl pivotal} for an event $A$ if changing the status of
$e$ changes whether or not $A$ occurs. The {\sl influence} of $e$ on the
event $A$, $I_A(e)$, is the probability that $e$ is pivotal for $A$.

\medskip
Next we need the definition of a ``flip time''.

\medskip\noindent
{\bf Definition:}
Given a graph and a vertex $x$, a time $t$ is called a {\it flip time for $x$}
if $x$ percolates at time $t$ but there is an edge $e$
which is pivotal for the event $\{\FCbetter(x,\infty)\}$
at time $t$ and which changes its status at time $t$.
(Note in this case, there is a $\delta >0$ such
that either (1) $x$ does not percolate during $(t-\delta,t)$
or (2) $x$ does not percolate during $(t,t+\delta)$.)

\begin{lemma}\label{l:prodfin}
In a spherically symmetric tree with spherically symmetric edge probabilities
$$
\Eb{W_n}\,\Pb{W_n=1}\le \Pb{W_n>0}^2.
$$
\end{lemma}

As we will later see in Lemma~\ref{l:piv}, the reverse inequality holds up to a multiplicative
constant under some reasonable assumptions.

\proof
Let $Q$ be the set of vertices in $T_n$ that are connected to $\rho$.
For $v\in T_n$, let
$\ev L_v$ denote the event that $v\in Q$
and $v$ is the leftmost vertex in $Q$.
Likewise, let $\ev R_v$ denote the event that $v\in Q$
and $v$ is the rightmost vertex in $Q$.
Then
$$
\Pb{Q=\{v\}}= \Pb{\ev L_v,\,\ev R_v}=
\frac{ \Pb{\ev L_v}\,\Pb{\ev R_v}}{\Pb{v\in Q}}\,,
$$
by the independence of what happens to the right of the
path from $\rho$ to $v$ and what happens to the left of this path.
Applying the arithmetic-geometric means inequality, we find
$$
\Pb{Q=\{v\}}^{1/2}\,\Pb{v\in Q}^{1/2}\le\frac 12\, \Pb{\ev L_v}+\frac 12 \,\Pb{\ev R_v}\,.
$$
When $Q\ne\emptyset$, there is precisely one vertex $v$ satisfying $\ev L_v$
and precisely one vertex satisfying $\ev R_v$. Hence, by summing the above over all
$v\in T_n$, we get
$$
\sum_{v\in T_n}
\Pb{Q=\{v\}}^{1/2}\,\Pb{v\in Q}^{1/2}\le\Pb{W_n>0}\,.
$$
Now note that for every $v\in T_n$ we have $\Pb{Q=\{v\}}=\Pb{W_n=1}/|T_n|$
and $\Pb{v\in Q}=\Eb{W_n}/|T_n|$.  The Lemma follows.
\QED

\proofof{Theorem~\ref{th:dyntrans}.(i)}
We will estimate from above the expected number of pivotal edges for the
event $\{\FCbetter(\rho,T_n)\}$ in a static configuration.
For each $m\in\{1,\dots,n\}$, let $v_m$ be the leftmost vertex in $T_m$,
and let $u(m,n)$ be the expected number of edges between $T_{m-1}$ and $T_m$
that are pivotal for $\{\FCbetter(\rho,T_n)\}$.
Also let $a(m,n)$ be the probability that $v_m$ is connected to $T_n$ within its subtree; that is,
$a(m,n)=\Pb{\FCbetterTree(v_m,T_n)}$.
To  estimate $u(m,n)$, we consider a different tree $T'$ which is
identical to $T$ until level $m$, but each vertex at level $m$ in $T'$
has only one child at level ${m+1}$, and the edge probability for the edges
between levels $m$ and $m+1$ in $T'$ is $a(m,n)=a(m,n;T)$
(and the $m+1$ level is the last level of $T'$).
The probability that the edge $[v_{m-1},v_{m}]$ is
pivotal for $\{\FCbetter(\rho,T_n)\}$ and $\FCbetter(\rho,T_n)$
holds is the probability that
in $T'$ the child of $v_m$ is the only vertex at level $m+1$ connected to $\rho$.
By Lemma~\ref{l:prodfin}, the latter is bounded by
$$
\Pb{\FCbetter(\rho,T_n)}^2\,\bl(|T_m|\, \ew m\,a(m,n)\br)^{-1}
$$
(where the notations all relate to the tree $T$).
Therefore,
$$
p_m\,u(m,n) \le
\bl(\ew m\,a(m,n)\br)^{-1}.
$$
Observe that the expected number of vertices $v\in T_k$
satisfying $\FCbetterTree(v_m,v)$ is $\ew k/\ew m$.
Therefore~\eqref{e:lyonsstrong}
applied to the tree $T^{v_m}$ gives
$$
a(m,n)^{-1} \asymp
\ew m
\sum_{k=m+1}^n{\ew k^{-1}}.
$$
Plugging this into the above, we get
\begin{equation}
\label{e.pumn}
p_m\,u(m,n) \le O(1)
\sum_{k=m+1}^n{\ew k^{-1}}.
\end{equation}

We now move to the dynamical setting. Let $Z_n$ be the set of times in
$[0,1]$ at which $\FCbetter(\rho,T_n)$, and let $Z=\bigcap_{n>0} Z_n$ be
the percolation times of the root in $[0,1]$.
It is clear that $\partial Z=\limsup_n \partial Z_n$.
(By definition, $\limsup_n A_n := \bigcap_{n>0} \overline{\bigcup_{j>n} A_j}$.)
Note that the set $\partial Z_n$ is the set of times at which a pivotal
edge for $\{\FCbetter(\rho,T_n)\}$ switches its value. Hence,
$$
\Eb{|\partial Z_n|}= \sum_{m=1}^n 2\,p_m\,(1-p_m)\,u(m,n)
\overset{\eqref{e.pumn}}
\le
O(1) \sum_{k=1}^n{k\,\ew k^{-1}}.
$$
Our assumptions therefore imply that
$\sup_n \Eb{|\partial Z_n|}<\infty$.
Consequently, $\liminf_{n\to\infty}|\partial Z_n|<\infty$ a.s.
Since $|\partial Z|\le \liminf_{n\to\infty} |\partial Z_n|$,
this proves (i) of Theorem~\ref{th:dyntrans}.
\QED

Part (ii) of Theorem~\ref{th:dyntrans} is an easy consequence of the following theorem.

\begin{theorem}\label{th:manyswitch}
Suppose that $\sup_j d_j<\infty$, \eqref{e:perc}
and the following assumptions hold:
\begin{align}
\label{e:fm}
&
\sum_{m=1}^n \frac{1}{m}\le O(1) \sum_{m=1}^n \sum_{k=m}^\infty \frac{1}{w_k}\,,
\\
\label{e:bbsum}
&
\sum_{n=0}^\infty\Bigl(\sum_{m=n+1}^\infty \frac{\ew n}{\ew m}\Bigr)^{-2}
<\infty\,,
\\
\label{e:finsmp}
&
\sum_{k=0}^\infty
\Bigl((k+1)\,\ew k \Bigl(\sum_{j=k}^\infty \ew j^{-1}\Bigr)^2\Bigr)^{-1}
<\infty
\,.
\end{align}
Then with positive probability there are infinitely many flip times for
the event $\{\FCbetter(\rho,\infty)\}$ in the time interval $[0,1]$.
\end{theorem}

Let $b_j$ denote the probability that a vertex at level $j$ percolates
to $\infty$ (at time $0$) through its leftmost child.

\begin{lemma}\label{l:bprod}
Assume $\sup_jd_j<\infty$,~\eqref{e:perc} and~\eqref{e:bbsum}.
Then
\begin{equation}
\label{e:bprod}
\prod_{j=0}^{n-1} (1-b_j)^{d_{j}-1}\asymp
\Bigl(\sum_{m=n}^\infty \frac{1}{\ew m}\Bigr)^2,
\end{equation}
where the implied constants may depend on the tree and on the
sequence $\{p_j\}$.
\end{lemma}

\proof
We start by deriving a rough estimate for $b_n$.
If $v\in T_n$ and $m> n$, then the expected
number of vertices $u\in T_m$ such that $\FCbetterTree(v,u)$ is
$\ew m/\ew n$. Therefore,~\eqref{e:lyonsstrong} gives
\begin{equation}
\label{e:bval}
b_n\asymp  \frac{1}{d_n\,\ew n}\,\Bigl(\sum_{m=n+1}^\infty \frac{1}{\ew m}\Bigr)^{-1}.
\end{equation}
This estimate in itself will not be fine enough to yield~\eqref{e:bprod},
but will be a useful first step.

For each node at level $j$ in the tree, we order its children
according to some fixed linear order (e.g., left to right,
if we think of the tree as embedded in the plane).
If $v$ is a vertex at level $n$ and $j\in\{1,\dots,n\}$,
let $u_j(v)$ denote the vertex at level $j$ that has
$v$ in its subtree, and let
$i_j(v)$ be the position of $u_j(v)$ among its siblings in the above order.
This induces an ordering on the vertices at level $n$:
we say that $v'<v$ if at the minimal $j$ such that
$i_j(v')\ne i_j(v)$ we have $i_j(v')<i_j(v)$.
Fix some $v\in T_n$.
Let $\ev L_v$ denote the event that $v$ is the minimal vertex
at level $n$ such
that $\rho$ percolates to $\infty$ through $v$.
Note that the probability that $v$ percolates to
$\infty$ within its subtree is $b_{n-1}/p_n$ and that
$\Pb{\FCbetter(\rho,v)}=\ew n/|T_n|$.
Hence
$$
\Pb{\ev L_v} = \frac{\ew n}{|T_n|}\,\frac{b_{n-1}}{p_n}\,
\prod_{j=0}^{n-1} (1-b_j)^{i_{j+1}(v)-1}.
$$
Since $\Pb{\percrt}=\sum_{v\in T_n}\Pb{\ev L_v}$, this gives
$$
\frac{p_n\,\Pb{\percrt}}{b_{n-1}\,\ew n} =
\frac{1}{|T_n|}\,
\sum_{v\in T_n}
\prod_{j=0}^{n-1} (1-b_j)^{i_{j+1}(v)-1}.
$$
We now use $|T_n|=\prod_{j=0}^{n-1} d_j$, and get
\begin{multline*}
\frac{p_n\,\Pb{\percrt}}{b_{n-1}\,\ew n}
=
\sum_{v\in T_n}
\prod_{j=0}^{n-1}\frac {(1-b_j)^{i_{j+1}(v)-1}}{d_{j}}
= \prod_{j=0}^{n-1}\sum_{i=1}^{d_{j}}  \frac {(1-b_j)^{i-1}}{d_{j}}
\\ 
= \prod_{j=0}^{n-1} \frac {1-(1-b_j)^{d_{j}}}{b_j\,d_{j}}
\,.
\end{multline*}
If we compare the factor corresponding to $j$ on the right with
$(1-b_j)^{(d_{j}-1)/2}$, we find that they agree up to a factor of $\exp\bl(O(b_j^2)\bigr)$,
where the implied constant may depend on $\sup_j d_j$ and
on $\sup_j b_j\le \sup_j p_j<1$.
Hence,
$$
\frac{p_n\,\Pb{\percrt}}{b_{n-1}\,\ew n}
=
\Bigl(\prod_{j=0}^{n-1} (1-b_j)^{(d_{j}-1)/2}\Bigr)
\exp\Bl(O(1)\sum_{j=0}^{n-1} b_j^2\Br).
$$
Now~\eqref{e:bprod} follows by squaring both sides,
using the estimate~\eqref{e:bval} for $b_{n-1}$, using
$p_n\,d_{n-1}\,\ew {n-1}=\ew n$ and noting that
$\sum_j b_j^2<\infty$ by~\eqref{e:bval} and~\eqref{e:bbsum}.
\QED

The following lemma can be seen as a partial converse to Lemma~\ref{l:prodfin},
but for convenience it is stated in a slightly different setting.

\begin{lemma}\label{l:piv}
Let $U_n$ denote the number of edges joining $T_{n-1}$ to $T_n$
through which $\rho$ percolates to $\infty$.
Then under the assumptions of Lemma~\ref{l:bprod}, we have
$$
\Pb{U_n=1}\Eb{U_n}\asymp 1\,.
$$
\end{lemma}
\proof
By~\eqref{e:bval} and~\eqref{e:bprod}, we have
$$
\prod_{j=0}^{n-1} (1-b_j)^{d_{j}-1}\asymp
(b_{n-1}\,\ew {n-1}\,d_{n-1})^{-2}=\Eb{U_n}^{-2}\,.
$$
Now multiply the left hand side by
$|T_n|\,p_1\,p_2\cdots p_{n-1}\,b_{n-1}$ and the
right hand side by its equal, $\Eb{U_n}$.
On the left hand side we then get $\Pb{U_n=1}$, as required.
\QED

\proofof{Theorem~\ref{th:manyswitch}}
The proof is based on a second moment argument.
For an edge $e$ let $X(e)$ denote the number of flips
(for $\FCbetter(\rho,\infty)$) occuring at times in $[0,1]$ when $e$ switches.
Let $m=m(e):=|e|$ denote the level of $e$; that is $e$
connects $T_m$ and $T_{m-1}$.
Set $\mathfrak X(e):=\One_{\{X(e)>0\}}$, $X_n:=\sum_{|e|\le n} X(e)$ and
$\mathfrak X_n:=\sum_{|e|\le n} \mathfrak X(e)$.
The second moment argument will be applied to $\mathfrak X_n$:
we will show that $\lim_{n\to\infty}\Eb{\mathfrak X_n}=\infty$,
and that $\sup_n \Eb{\mathfrak X_n^2}/\Eb{\mathfrak X_n}^2 <\infty$.

At this point, we use an equivalent version of the dynamics in which at rate 1, an edge
is {\it refreshed} and when refreshed, it chooses to be in state 1 with probability $p_e$.
Let now $Y_e$ be the set of times in which $e$ refreshed, and
let $A_e$ be the set of times $t\in[0,1]$ at which $e$ is pivotal for
$\{\FCbetter(\rho,\infty)\}$. Since $2\,p_m\,(1-p_m)$ is
the probability a refresh time is a switch time, and
 $Y_e$ is a Poisson point process with rate $1$ independent from $A_e$,
we have
$$
1-\exp\bigl(-\mu(A_e)\bigr) \ge
\Eb{\mathfrak X(e)\md A_e}\ge 2\,p_m\,(1-p_m)\,\bigl(1-\exp\bigl(-\mu(A_e)\bigr)\bigr),
$$
where $\mu$ denotes Lebesgue measure.  It follows that
$$
\Eb{\mathfrak X(e)\md A_e} \asymp \mu(A_e)\,.
$$
Moreover, Fubini gives
$$
\Eb{\mu(A_e)}=\Pb{e\text{ pivotal for $\{\FCbetter(\rho,\infty)\}$ at time }0}\,.
$$
Hence,
$$
\Eb{\mathfrak X_n}\asymp
\sum_{m=1}^n\sum_{|e|=m}\Pb{e\text{ pivotal for $\{\FCbetter(\rho,\infty)\}$ at time }0}\,.
$$
Lemma~\ref{l:piv} easily implies that if $|e|=m$, then
$$
\Pb{e\text{ pivotal for $\{\FCbetter(\rho,\infty)\}$ at time } 0}
\asymp \frac{1}{|T_m|\,\ew {m-1}\,d_{m-1}\,b_{m-1}}\,.
$$
The above together with~\eqref{e:bval} gives
\begin{equation}
\label{e:EXn}
\Eb{\mathfrak X_n}
\asymp
\sum_{m=1}^n
\sum_{k=m}^\infty \frac{1}{\ew k}\,.
\end{equation}

We now turn to estimating $\Eb{\mathfrak X_n^2}$.
Let $e,e'$ be two different edges at levels $m$ and $m'$, respectively,
where $m,m'\le n$.
Then $X(e)\,X(e')\le |Y_e\cap A_e|\cdot|Y_{e'}\cap A_{e'}|$.
Let $\nu_{e,e'}$ denote the counting measure on the
set $(Y_e\cap A_e)\times (Y_{e'}\cap A_{e'})\subseteq[0,1]^2$,
and let $I,I'\subseteq[0,1]$ be disjoint time intervals.
Note that $Y_e\cap I$, $Y_{e'}\cap I'$ and
$(A_e\cap I,A_{e'}\cap I')$ are independent.
(Note however that $A_e\cap I$ is usually not independent from $A_{e'}\cap I'$.)
Therefore
$$
\EB{\nu_{e,e'}(I\times I')\md
  A_e\cap I,\,A_{e'}\cap I'}
=
\mu(A_e\cap I)\, \mu(A_{e'}\cap I')\,.
$$
Hence
$$
\Eb{\nu_{e,e'}(I\times I')}=
\int_{I\times I'}\Pb{t\in A_e,\,s\in A_{e'}}\,dt\,ds\,.
$$
For $e\neq e'$, $\nu_{e,e'}$ gives no mass to the diagonal, and hence we can 
conclude that
$$ 
\Eb{X(e)\,X(e')} \le \Eb{\nu_{e,e'}([0,1]\times [0,1])}=
\int_0^1\int_0^1\Pb{t\in A_e,\,s\in A_{e'}}\,dt\,ds\,.
$$ 
Since
$\sum_{|e|\le n}\mathfrak X(e)\,\mathfrak X(e)=\mathfrak X_n$,
we have
$$
\mathfrak X_n^2=
\mathfrak X_n+\sum_{|e|,|e'|\le n}\One_{\{e\ne e'\}}\,
\mathfrak X(e)\,\mathfrak X(e')
\le
\mathfrak X_n+\sum_{|e|,|e'|\le n}\One_{\{e\ne e'\}}\,
X(e)\,X(e')
\,.
$$
Consequently,
$$
\Eb{\mathfrak X_n^2}
\le
\Eb{\mathfrak X_n}+
\sum_{|e|,|e'|\le n}\One_{\{e\ne e'\}}\,
\int_0^1\int_0^1\Pb{t\in A_e,\,s\in A_{e'}}\,dt\,ds\,.
$$

At this point, we break up the pairs $(e,e')$ for which $e\neq e'$ into two sets, those where
$e$ and $e'$ don't lie on the same path from the root to $\infty$ (which is the generic case)
and those where they do lie on the same path. Call the first class $\mathcal E_1$ and
the second class $\mathcal E_2$. We consider now pairs $(e,e')$ in $\mathcal E_1$.

Let $v_0=\rho,v_1,\dots,v_{m}$ denote the path
from the root $\rho$ to the endpoint of $e$ at level $m=|e|$,
and let
$v'_0,v'_1,\dots,v'_{m'}$ denote the path
from the root to the endpoint of $e'$ at level $m'=|e'|$.
Let $k\le (m-1)\wedge (m'-1)$ be maximal such that $v_k=v'_k$.
Also, fix $s,t\in[0,1]$ and set $r:=|s-t|$.
Note that for every $j\in\N_+$ and any edge at level $j$,
the probability that the edge is open at time $s$ and at time $t$ is
$p_j^2+(1-p_j)\,p_j\,\exp(-r)$.
For $j=0,\dots,m-1$, let $\ev U_j$ denote the event that at time $t$ we have
$\FCbetter(v_j,\infty)$ inside $T^{v_j}\setminus v_{j+1}$,
and let $\ev U_j'$ denote the corresponding event with
each $v_i$ replaced by $v_i'$, with $t$ replaced by $s$ and with
$m$ replaced by $m'$.
Note that the event $\{t\in A_e,\,s\in A_{e'}\}$ is contained
in the intersection of the following events:
$\ev L:=\{\FCC(\rho,t,v_k),\,\FCC(\rho,s,v_k)\}$,
$\ev Q_1:=\{\FCCTree(v_k,t,v_{m-1})\}$,
$\ev Q_1':=\{\FCCTree(v_k,s,v'_{m'-1})\}$,
$\ev Q_2:=\{\FCCTree(v_m,t,\infty)\}$,
$\ev Q_2':=\{\FCCTree(v'_{m'},s,\infty)\}$,
$\ev Z_1:=\bigcap_{j=0}^{k-1}\neg \ev U_j$,
$\ev Z_2:=\bigcap_{j={k+1}}^{m-1}\neg \ev U_j$,
$\ev Z'_2:=\bigcap_{j={k+1}}^{m'-1}\neg \ev U'_j$,
and that these events are all independent.
Consequently,
\begin{multline*}
\Pb{t\in A_e,\,s\in A_{e'}}
\le
\prod_{j=1}^k\bigl( p_j^2+(1-p_j)\,p_j\,\exp(-r)\bigr)\times
\prod_{j=k+1}^{m-1} p_j\times
\prod_{j=k+1}^{m'-1} p_j\times
{} \\ {}
\frac{b_{m-1}}{p_m}\times
\frac{b_{m'-1}}{p_{m'}}\times
\prod_{j=0}^{k-1}(1-b_j)^{d_j-1}\times
\prod_{j=k+1}^{m-1}(1-b_j)^{d_j-1}\times
\prod_{j=k+1}^{m'-1}(1-b_j)^{d_j-1}.
\end{multline*}
Setting $\delta:=1-\sup_j p_j$ and noting that
$r\le 1$, we may estimate the first product as
$$
\le
\bigl(1-\delta\,r/3)^k
\prod_{j=1}^k p_j
\le \exp\Bigl(-\frac{\delta\,k\,r}3\Bigr)\prod_{j=1}^k p_j\,.
$$
Using the above and Lemma~\ref{l:bprod}, we arrive at the estimate
\begin{multline*}
\Pb{t\in A_e,\,s\in A_{e'}}
\le
O(1)\,\exp\Bigl(-\frac{\delta\,k\,r}3\Bigr)\times {}\\ {}
\frac{\Bigl(\prod_{j=1}^{m-1}p_j\Bigr)\Bigl( \prod_{j=1}^{m'-1}p_j\Bigr)
 b_{m-1}\, b_{m'-1}
 \Bigl(\sum_{j=m}^\infty \ew j^{-1}\Bigr)^2
 \Bigl(\sum_{j=m'}^\infty \ew j^{-1}\Bigr)^2
}
{\Bigl(\prod_{j=1}^k p_j\Bigr) (1-b_k)^{2d_k-2}
 \Bigl(\sum_{j=k}^\infty \ew j^{-1}\Bigr)^2
}\,.
\end{multline*}
Since we are assuming $\sup_j d_j<\infty$ and since $b_j\le p_{j+1}\le 1-\delta$,
we have $(1-b_k)^{2-2d_k}=O(1)$, and that factor may be dropped.
Now note that when $(t,s)$ is uniform in $[0,1]^2$, the probability that $r$
is in any interval $I\subseteq [0,1]$ is at most twice the length of $I$.
Since $\int_0^1\exp(-\delta\,k\,r/3)\,dr\le O\bl(1/(\delta\, (k+1))\br)=O(1/(k+1))$,
we get
\begin{multline*}
\int_0^1\int_0^1
\Pb{t\in A_e,\,s\in A_{e'}}\,dt\,ds \le
\\
O(1)\,
\frac{\Bigl(\prod_{j=1}^{m-1}p_j\Bigr)\Bigl( \prod_{j=1}^{m'-1}p_j\Bigr)
 b_{m-1}\, b_{m'-1}
 \Bigl(\sum_{j=m}^\infty \ew j^{-1}\Bigr)^2
 \Bigl(\sum_{j=m'}^\infty \ew j^{-1}\Bigr)^2
}
{(k+1)\Bigl(\prod_{j=1}^k p_j\Bigr)
 \Bigl(\sum_{j=k}^\infty \ew j^{-1}\Bigr)^2
}\,.
\end{multline*}
If we fix $m,m'$ and $v_k$, there are at most $|T_m|/|T_k|$ possible
choices for $e$ and $|T_{m'}|/|T_k|$ possible choices for $e'$.
Thus, there are at most $|T_m|\,|T_{m'}|\,|T_k|^{-2}$ possible
choices for pairs $(e,e')$. Since $|T_j|=d_{j-1}\,|T_{j-1}|$
and $T_j\,\prod_{i=1}^jp_i=\ew j$, the sum of the above over
all such pairs $(e,e')$ is
\begin{multline*}
\le O(1)\,
\frac{d_{m-1}\,d_{m'-1}\,\ew {m-1}\,\ew {m'-1}\,
 b_{m-1}\, b_{m'-1}
 \Bigl(\sum_{j=m}^\infty \ew j^{-1}\Bigr)^2
 \Bigl(\sum_{j=m'}^\infty \ew j^{-1}\Bigr)^2
}
{(k+1)\,|T_k|\,\ew k
 \Bigl(\sum_{j=k}^\infty \ew j^{-1}\Bigr)^2
}
\\
\stackrel{\eqref{e:bval}}\asymp
\frac{
 \Bigl(\sum_{j=m}^\infty \ew j^{-1}\Bigr)
 \Bigl(\sum_{j=m'}^\infty \ew j^{-1}\Bigr)
}
{(k+1)\,|T_k|\,\ew k
 \Bigl(\sum_{j=k}^\infty \ew j^{-1}\Bigr)^2
}\,\,\, .
\end{multline*}
We now sum over all possible choices for $v_k$, which eliminates
the $|T_k|^{-1}$ factor.
Next, we bound the sum of the resulting expression for $m\in\{k+1,k+2,\dots,n\}$
and $m'\in\{k+1,k+2,\dots,n\}$ by summing over all $m,m'=1,2,\dots,n$.
Finally, we sum over $k=0,1,\dots,n-1$,
to obtain
\begin{multline*}
\sum_{|e|,|e'|\le n} \One_{\{(e,e')\in \mathcal E_1\}}\,
\int_0^1\int_0^1
\Pb{t\in A_e,\,s\in A_{e'}}\,dt\,ds
\\
 \le
O(1)\left( \sum_{m=1}^n\sum_{j=m}^\infty \ew j^{-1}\right)^2
\sum_{k=0}^\infty
\Bigl((k+1)\,\ew k \Bigl(\sum_{j=k}^\infty \ew j^{-1}\Bigr)^2\Bigr)^{-1}.
\end{multline*}
By~\eqref{e:finsmp} and~\eqref{e:EXn},
this is at most $O(1)\,\Eb{\mathfrak X_n}^2$.

We now explain the necessary modifications for the case
$(e,e')\in\mathcal E_2$. Let $m=|e| < |e'|=m'$.
Using the same notations as above,
it is easy to see that
the event $\{t\in A_e,\,s\in A_{e'}\}$ is contained
in the intersection of the following independent events:
$\{\FCC(\rho,t,v_{m-1}),\,\FCC(\rho,s,v_{m-1})\}$,
$\{\FCCTree(v_m,s,v_{m'-1})\}$,
$\{\FCCTree(v_{m'},s,\infty)\}$ 
and $\bigcap_{j={0}}^{m'-1}\neg \ev U'_j$.
This leads, after a computation exactly as before, to
\begin{multline*}
\int_0^1\int_0^1
\Pb{t\in A_e,\,s\in A_{e'}}\,dt\,ds
\\
 \le O(1)\,
\frac{\Bigl( \prod_{j=1}^{m'-1}p_j\Bigr)\, b_{m'-1}
 \prod_{j=0}^{m'-1}(1-b_j)^{d_j-1}
}
{m+1
}\,.
\end{multline*}
With $e$ and $m'$ fixed, there are at most $|T_{m'}|/|T_m|$ possible
choices for $e'$ and so the sum of the above over such $e'$ is at most
$$
O(1)\,\frac{ w_{m'-1}\, b_{m'-1} \prod_{j=0}^{m'-1}(1-b_j)^{d_j-1}}
{m |T_m|}\le
O(1)\,\frac{ \sum_{k=m'}^\infty \frac{1}{w_k}}
{m |T_m|}\,,
$$
by \eqref{e:bprod} and \eqref{e:bval}. At level $m$, there are $|T_m|$ choices for $e$. As
$m'\ge m+1$, we can sum over $m'$ from 1 to $n$ and then sum over $m$ from 1 to $n$ to yield
\begin{multline*}
\sum_{|e|,|e'|\le n} \One_{\{(e,e')\in \mathcal E_2\}}\,
\int_0^1\int_0^1
\Pb{t\in A_e,\,s\in A_{e'}}\,dt\,ds
\\
 \le
O(1)\left( \sum_{m=1}^n\sum_{j=m}^\infty \ew j^{-1}\right)
\sum_{m=1}^n \frac{1}{m} 
\\=
O(1)\left( \sum_{m=1}^n\sum_{j=m}^\infty \ew j^{-1}\right)^2
\frac{\sum_{m=1}^n \frac{1}{m}}
{\left( \sum_{m=1}^n\sum_{j=m}^\infty \ew j^{-1}\right)}.
\end{multline*}
By~\eqref{e:fm} and~\eqref{e:EXn}, this is also at most $O(1)\,\Eb{\mathfrak X_n}^2$.

All of the above therefore yields
$\Eb{\mathfrak X_n^2} \le \Eb{\mathfrak X_n}+O(1)\,\Eb{\mathfrak X_n}^2$.
Since $\lim_{n\to\infty}\Eb{\mathfrak X_n}=\infty$ by~\eqref{e:fm}
and~\eqref{e:EXn}, this gives
$\Eb{\mathfrak X_n^2} \le O(1)\,\Eb{\mathfrak X_n}^2$.
A one-sided Chebyshev inequality (see, e.g., Lemma 5.4 in \cite{HPS}) 
or alternatively the Paley Zygmund inequality yields
that there is some $c>0$, which does not depend on
$n$, such that $\PB{\mathfrak X_n\ge c\,\Es{\mathfrak X_n}}\ge c$.
Hence
$\Pb{\lim_{n\to\infty}\mathfrak X_n=\infty}\ge c$,
which completes the proof.
\QED

\proofof{Theorem~\ref{th:dyntrans}.(ii)}
This easily follows from Theorem~\ref{th:manyswitch}.
\QED

\section{Proof of Theorem \ref{th:pivotal}}\label{sec:6}

We start with a lemma connecting the concepts of flip time and influence.

\begin{lemma} \label{lemma:flip=pivotal}
Fix a vertex $x$. Then
$$
2\sum_e I_x(e)p_e(1-p_e)=  E[|{\mathcal S}|],
$$
where ${\mathcal S}$ is the set of flip times for $x$ during $[0,1]$.
\end{lemma}

\proof
Fix $e$. The probability that during $[t,t+dt]$ the edge $e$ switches its state
precisely once is easily seen to be $2\,p_e\,(1-p_e)\,dt+O(dt^2)$. 
Conditioning on that time, the probability that $e$ is pivotal for 
$\{\FCbetter(x,\infty)\}$ at that time is $I_x(e)$.
Hence, the probability that there is a flip associated to $e$ during $[t,t+dt]$
is $2\,I_x(e)\,p_e\,(1-p_e)\,dt+O(dt^2)$. It follows that
$E[{\mathcal S}_e]= 2\,I_x(e)\,p_e\,(1-p_e)$ where
${\mathcal S}_e$ is the set of flip times associated to $e$ during $[0,1]$.
Summing over $e$ yields the result.
\QED

\proofof{Theorem~\ref{th:pivotal}}
Fix $x$.
Let ${\mathcal E}_n$ be the set of edges which are within graph distance
$n$ of $x$ and let ${\mathcal F}_n$ be the $\sigma$-algebra generated
by the evolution of the edges in ${\mathcal E}_n$ during
the time interval $[0,1]$. Let
$$
X_n(t)=
X_n(\omega,t):= P(\FCC(x,t,\infty)|{\mathcal F}_n).
$$
While conditional probabilities are usually only defined a.s.,
it is clear that there is a canonical version of these
conditional probabilities and these will always be used.
Let $V_n$ denote the total variation of $X_n(t)$ on $[0,1]$.

The following two lemmas are left to the reader.

\begin{lemma} \label{lemma:totalvariation}
$$
E[V_n]=2\sum_{e\in {\mathcal E}_n}I(e) p_e(1-p_e).
$$
\end{lemma}

\begin{lemma} \label{lemma:Vmartingale}
$\{V_n\}_{n\ge 1}$ is a submartingale.
\end{lemma}

By our assumption~\eqref{e:piv} and by Lemma~\ref{lemma:totalvariation},
we have $\sup_n E(V_n)<\infty$.
Since $\{V_n\}_{n\ge 1}$ is a nonnegative submartingale, this implies that
there is an a.s.\ limit $V:=\lim_{n\to\infty} V_n$ satisfying
$E(V)<\infty$.
Now, for all $t$, the Martingale convergence theorem tells us that
$X_n(t)$ converges a.s.\ to $\One_{\{\FCC(x,t,\infty)\}}$.
By Fubini's theorem, for a.e.\ $\omega$, there exists
$A_\omega\subseteq [0,1]$ such that $\mu(A_\omega)=1$
($\mu$ is Lebesgue measure here) and
\begin{equation} \label{e:**}
\lim_{n\to\infty}X_n(\omega,t)=
\One_{\{\FCC(x,t,\infty)\}} \mbox{ for all } t\in A_\omega.
\end{equation}
Now define
$$
\tilde{X}(\omega,t):=
\left\{ \begin{array}{ll}
\One_{\{\FCC(x,t,\infty)\}}
&\text{if }t\in A_\omega,\\
\limsup_{s \uparrow t, s\in A_\omega}
\One_{\{\FCC(x,s,\infty)\}}
&\text{if }t\not\in A_\omega.
\end{array}
\right.
$$
Statement (\ref{e:**}) implies that the total variation of
$\tilde{X}$ restricted to time points in $A_\omega$
is at most $V$ for a.e.\ $\omega$. It is then easy to check that
the total variation of $\tilde{X}$ over $[0,1]$ is then
at most $V$ for a.e.\ $\omega$ as well. We conclude that a.s.\
$\One_{\{\FCC(x,t,\infty)\}}$ is equal a.s.\ to a function of bounded variation.

We now show that the fact that a.s.\ $\One_{\{\FCC(x,t,\infty)\}}$ is
equal a.e.\ to a function of bounded variation implies that there are no
exceptional times.
Let $X$ be the Lebesgue measure of the amount of time that $x$
percolates during $[0,1]$. By Fubini's theorem, $E(X)$
is the probability that $x$ percolates. It follows that with positive
probability, $X> 0$. If there were exceptional times of nonpercolation,
an easy application of Kolmogorov's 0-1 law tells us that a.s.\ there
would be such times in every nonempty interval. However,
the latter together with the fact that the set of times at which $x$ does not
percolate is open and that $X> 0$ contradicts the fact that
$\One_{\{\FCC(x,t,\infty)\}}$ is equal a.s.\ to a function of bounded
variation.
\QED

\section{A 0-1 Law} \label{sec:01law}
In this section, we present a 0-1 law concerning the process.
In addition to being of interest in itself, we believe it might be
useful for obtaining a better understanding of the path behavior of our
process and might be relevant to some of the problems at the end of the paper.

\begin{theorem}\label{t.01}
Consider dynamical percolation $(\omega_t:t\in\R)$ on a spherically symmetric tree $\tree$
with spherically symmetric edge probabilities,
and let $\perctime$ be the set of times $t\in\R$ such that the cluster of the
root is infinite in $\omega_t$.
If $\Pb{0\in\partial \perctime}>0$, then a.s.\ $\perctime=\partial \perctime$
(and hence by Lemma~\ref{lemma:staysclose} there is a.s.\ a dense set of
times $t\in\R$ in which there is no infinite cluster in $\omega_t$).
\end{theorem}

Now consider an arbitrary locally finite tree $\tree$ with root $\rho$ and a vertex 
$v$ of $\tree$.
For any $\omega\subseteq 2^{E(\tree)}$, we may start dynamical percolation
$\omega_t$ with $\omega_0=\omega$.
It is easy to see that for this Markov process, the probability that there is a positive
$\epsilon$ such that $\FCCTree(v,t,\infty)$ for all times $t\in[0,\epsilon)$
is $0$ or $1$. Let $h_v(\omega)\in\{0,1\}$ denote this probability.

\begin{lemma}\label{l.hered}
With the above notation, let $v_1,\dots,v_m$ denote the children of $v$; that is,
the neighbors of $v$ within $\tree_v$.
Then
$$
h_v(\omega) = \max\bigl\{ 1_{[v,v_j]\in\omega}\,h_{v_j}(\omega):j=1,2,\dots,m\bigr\}
$$
holds for a.e.\ $\omega$ with respect to the invariant measure of the Markov process
$\omega_t$.
\end{lemma}

We point out that the lemma does not need to assume that $\tree$ is spherically symmetric.

\proof
It is certainly clear that $h_v$ is at least as large as the $\max$ on the right hand side.
We therefore only need to prove the reverse inequality.
Let $U_j$ be the set of times $t\in[0,\infty)$ such that
$v$ does not percolate to $\infty$ in $[v,v_j]\cup\tree^{v_j}$ at time $t$.
Then $U_j$ is a relatively open set.

Set $\WW_k:=\bigcap_{j=1}^k U_j$,
and $\WW_k':=\bigcup_{j=1}^k \bigl([0,\infty)\setminus\closure{U_j}\bigr)$.
Note that the $\max$ on the right hand side in the statement of the lemma is
equal to $1_{0\in \WW'_m}$.
We prove by induction on $k$ that $0\in\closure{\WW_k}\cup \WW_k'$ a.s.\ holds for $k=0,1,\dots,m$.
The case $k=m$ then implies the statement of the lemma.
The base of the induction, $k=0$, is clear, because $\WW_0=[0,\infty)$, by convention.
Now suppose that $0<k<m$ and $0\in\closure{\WW_k}\cup \WW_k'$.
If $0\in \WW_k'$, then $0\in \WW_{k+1}'$.
Therefore, suppose that $0\in \closure{\WW_k}$.
Hence, there is a sequence $(t_n:n\in\N)$ in $\WW_k$
such that $t_n\to 0$. Moreover, it is easy to see that we may choose
the sequence to depend only on $\WW_k$
and in such a way that each $t_n$ is measurable.
In particular, the sequence $\{t_n\}$ is independent from the restriction of
$(\omega_t:t\ge 0)$ to $[v,v_{k+1}]\cup \tree^{v_{k+1}}$.
Fix some $n\in\N$, and suppose for the moment that $t_n$ is in the closure of $U_{k+1}$.
Then we can find a point $t'$ in $U_{k+1}$ arbitrarily close to $t_n$.
Since $t_n\in \WW_k$, and $\WW_k$ is relatively open,
there is a point $t'$ arbitrarily close to $t_n$ that is in $\WW_{k+1}=\WW_k\cap U_{k+1}$.
Therefore, in the case that
$\bl\{n:t_n\in\closure{U_{k+1}}\br\}$ is infinite a.s.,
we have $0\in\closure{\WW_{k+1}}$ a.s.\ and
the inductive claim follows.

For every measurable $S \subseteq [0,1]$ we have by elementary Fourier analysis
that $1_S(t)-1_S(t+t_n)$ tends to zero
in $L^2$ as $n\to\infty$. Therefore, there is some infinite $Y\subseteq\N$ such that $1_S(t)-1_S(t+t_n)$
tends to zero a.e.\ as $n\to\infty$ within $Y$.  Consequently, a.e.\ $t\in S$
satisfies $\bigl|\{n:t+t_n\in S\}\bigr|=\infty$.
We may apply this to the set $S:=\closure{U_{k+1}}\cap[0,1]$. 
However, given the sequence $\{t_n\}$,
the distribution of $\closure{U_{k+1}}$ is invariant under translations.
Consequently, a.s.\ either $0\in \WW_{k+1}'$ or $\bigl|\bl\{n:t_n\in \closure {U_{k+1}}\br\}\bigr|=\infty$.
This proves $0\in \closure{\WW_{k+1}}\cup \WW_{k+1}'$ a.s., and completes the induction.
The statement of the lemma follows immediately.
\QED

\begin{lemma}\label{l.wake}
Consider stationary percolation on
a spherically symmetric tree with spherically symmetric
edge probabilities (and, as usual, assume that the edge probabilities
are bounded away from $0$ and $1$).
Then a.s.\ $W_\infty=\lim_{n\to\infty} W_n/\ew n$ exists and $W_\infty<\infty$.
Moreover, a.s.\ $W_\infty>0$ if and only if $\percrt$.
\end{lemma}
\proof
As we have noted before, $W_n/\ew n$ is a non-negative martingale,
which implies the a.s.\ existence and finiteness of $W_\infty$.
Let $X_n$ be the set of vertices $v$ at level $n$ satisfying
$\FCbetter(\rho,v)$, and
let $U_n:=\{v\in X_n: \FCbetterTree(v,\infty)\}$.
Fix some $v\in T_n$. 
With no loss of generality, assume that
$\Pb{\percrt}>0$, and hence $\Pb{v\in U_n}>0$.
For $m\ge n$, let $X^v_m:=\{u\in T_m:\FCbetterTree(v,u)\}$ and
$W^v_m:=|X^v_m|$.
The inequality~\eqref{e:L2bounded} applied to $T^v$ implies that
there is a universal constant $\delta>0$ such that
$$
\PB{W^v_m\ge \delta\,\Es{W^v_m\md W^v_m>0}\md W^v_m>0}\ge \delta\,.
$$
Since $\One_{\{W^v_m>0\}}\to \One_{\{\FCbetterTree(v,\infty)\}}$ a.s.\ as $m\to\infty$,
and $\Eb{W_m^v\md W_m^v>0}\ge \ew m/|T_n|$,
this implies
$$
\liminf_{m\to\infty} \PB{W^v_m \ge \delta \ew m\,|T_n|^{-1}\md \FCbetterTree(v,\infty)}\ge \delta\,.
$$
Hence,
$$
\PB{\lim_{m\to\infty} W^v_m/\ew m>0\md \FCbetterTree(v,\infty)}\ge \delta\,.
$$
By conditioning on the set $U_n$ and using conditional independence on the
various trees $T^v$, $v\in U_n$, we therefore get
$$
\Pb{W_\infty>0\md U_n} \ge 1-(1-\delta)^{|U_n|}\,.
$$
By Lemma 4.2 in~\cite{PPray}, a.s.\ on the
event $\percrt$ we have $\lim_{n\to\infty} |U_n|=\infty$.
Hence, for every finite $N$ we have $\Pb{|U_n|>N\md\percrt}\to 1$
as $n\to\infty$.
The lemma follows.
\QED

\proofof{Theorem~\ref{t.01}}
Let $\omega$ be a sample from the stationary measure of the Markov process $\omega_t$.
Let $q_n := \Eb{h_{u_n}(\omega)}$, where $u_n$ is a vertex at level $n$
(since the tree is spherically symmetric, the choice of $u_n$ does not affect $q_n$).
Let $\ev F_n$ denote the $\sigma$-field generated by the restriction of
$\omega$ to the ball of radius $n$ about the root $u_0$.
Lemma~\ref{l.hered} easily implies by induction
that $h_{u_0}(\omega)=1$ if and only if there is a vertex $v$ at level $n$ that is
connected in $\omega$ to $u_0$ and satisfies $h_{v}(\omega)=1$.
Therefore,
$$
\Eb{h_{u_0}(\omega) \md \ev F_n} = 1-(1-q_n)^{W_n} = 1-\exp\bl( \log(1-q_n)\, W_n\br).
$$
Since $\Eb{h_{u_0}(\omega) \md \ev F_n}$ tends to $h_{u_0}(\omega)$ as $n \to \infty$, we conclude that a.s.\
$\log(1-q_n)\,W_n$ tends to $0$ or $-\infty$.
If 
$$
\PB{\lim_{n\to\infty}\log(1-q_n)\,W_n=-\infty}>0\,,
$$
then Lemma~\ref{l.wake} implies
$$
\PB{\lim_{n\to\infty}\log(1-q_n)\,W_n=-\infty\md\percrt}=1\,.
$$
Therefore, we get either $h_{u_0}(\omega)=0$ a.s.,
or else $h_{u_0}(\omega)=\One_{\{\percrt\}}$ a.s.
The theorem follows.
\QED

\section{Some open questions} \label{sec:questions}

Following are a few questions and open problems suggested by the
present paper.

\newcounter{saveenum}
\begin{enumerate}
\item  In the spherically symmetric tree case, if $\ew k\asymp k^2$, is it the case that
with positive probability the set of times $t\in[0,1]$
at which the root percolates has infinitely many connected components?
In this case $\Eb{\mathfrak X_n}\asymp \log n$
grows to $\infty$ but the second moment method fails.

\item Under the assumption of Theorem \ref{th:pivotal}, is it the case that
$\{t\in[0,1]:\FCC(\rho,t,\infty)\}$
has finitely many
connected components a.s.? (From an earlier remark, this would be true if
in this setting
finiteness of the left-hand term in~\eqref{eq:Fatou} implies finiteness of the
right-hand term.)

\item Does the conclusion of Theorem \ref{th:sstrees1}(ii) hold
under the weaker assumptions
that
$$
\limsup_n \frac{w_n}{n(\log n)^\alpha} < \infty
$$
for some $\alpha \le 2$
and the tree percolates with positive probability at a fixed time?
We describe a natural approach which does {\it not} work.
Note that under the above assumption, one can find a new tree
which dominates the original tree (in the sense that the number of vertices at the $n$th
level level is larger for any $n$) for which the new $w_n$'s satisfy the assumption
of Theorem \ref{th:sstrees1}(ii) and hence would have exceptional times. In \cite{PP},
it is shown that this domination has a number of implications.
However, one cannot conclude that the set of times at which the original tree
percolates is dominated by the set of times at which the new tree percolates.
An example is  $T_1$ being a tree with degrees
$d_1=1$ and $d_2=2$, $T_2$ being a tree with degrees $d_1=2$ and $d_2=1$ and the
edge probabilities are very small. Then $T_2$ dominates $T_1$, but the probability that
the root is connected to level 2 throughout the time interval $[0,10]$ is larger for $T_1$.
\setcounter{saveenum}{\value{enumi}}
\end{enumerate}

There are various questions concerning the path behavior of the process
which might be interesting to pursue.
In the following questions, we consider a spherically symmetric tree in which
the root percolates with positive probability at a fixed time.
Let $Z:=\{t\in\R:\FCC(\rho,t,\infty)\}$.

\begin{enumerate}
\setcounter{enumi}{\value{saveenum}}
\item
Are the boundary points of the connected components of $\R\setminus Z$
always flip times?

\item
If $Z$ has connected components of positive length, do the boundary points
of these intervals have to also be boundary points of intervals in 
$\R\setminus Z$?

\item If there are exceptional times of nonpercolation,
is $Z$ the closure of the flip times?

\end{enumerate}

\bigskip
\noindent{\bf Acknowledgments}.
We would like to express our appreciation to Zhan Shi. 
He contributed in different ways at the beginning stages of this project,
helping with an earlier version of Theorem \ref{th:sstrees1}.(i)
and with constructing some initial examples. 
Research supported in part by
NSF grant DMS-0605166 (Peres) and
the Swedish Natural Science Research Council
and the G\"{o}ran Gustafsson Foundation (KVA) (Steif).
J.S. thanks Paris VI and Microsoft for hospitality during which parts of this
work were completed. Some of this work was also carried out at the
Park City Mathematics Institute.

\bigskip
\filbreak
\begingroup
{
\small
\parindent=0pt

Yuval Peres\\
Microsoft Corporation\\
One Microsoft Way\\
Redmond, WA 98052, USA\\
{peres@math.berkeley.edu} \\
{\tt
{http://stat-www.berkeley.edu/\string~peres/}  }

\bigskip

Oded Schramm\\
Microsoft Corporation\\
One Microsoft Way\\
Redmond, WA 98052, USA\\
{\tt
{http://research.microsoft.com/\string~schramm/}}

\bigskip

Jeffrey E. Steif\\
Mathematical Sciences \\
Chalmers University of Technology \\
and \\
Mathematical Sciences \\
G\"{o}teborg University \\
SE-41296 Gothenburg, Sweden \\
{steif@math.chalmers.se} \\
{\tt {http://www.math.chalmers.se/\string~steif/}}

}

\filbreak

\endgroup

\end{document}